\documentclass[a4paper,reqno]{amsart}

\usepackage{amsmath,amsfonts,amssymb,amsthm,bm}
\usepackage{graphicx,psfrag,color}

\usepackage[colorlinks,bookmarks,linkcolor=black,citecolor=black]{hyperref} 
\usepackage{enumitem}
\usepackage{textcase}

\def\bfalpha{\bm{\alpha}}

\def\subsc#1{\textsc{\MakeTextLowercase{#1}}} 
\def\sublem#1{\subsc{L}\text{\tiny\ref{#1}}}

\setlist{itemsep=3pt,parsep=0pt,topsep=2pt,partopsep=0pt}  
\setlist{leftmargin=*,itemindent=\parindent} 
\setenumerate{itemindent=0pt}

\def\itm#1{\rm ({#1})} 
\def\itmit#1{\itm{\it #1\,}} 
\def\rom{\itmit{\roman{*}}} 
\def\abc{\itmit{\alph{*}}}

\newcommand{\By}[2]{\overset{\mbox{\tiny{#1}}}{#2}} 
\newcommand{\ByRef}[2]{   \By{\eqref{#1}}{#2} }

\newcommand{\geBy}[1]{    \By{#1}{\ge} }

\newcommand{\leByRef}[1]{ \ByRef{#1}{\le} } 
\newcommand{\geByRef}[1]{ \ByRef{#1}{\ge} }

\newcommand{\kreg}{k_{\mathrm{reg}}}

\renewcommand{\subset}{\subseteq}

\newcommand{\eps}{\varepsilon}

\newcommand{\dcup}{\dot\cup}

\newcommand{\degt}{\tilde{\mathrm{d}}}
\newcommand{\degen}{\mathrm{degen}}

\newcommand{\EMAIL}[1]{  \textit{E-mail}: \texttt{#1}}

\newtheorem{theorem}{Theorem}
\newtheorem{lemma}[theorem]{Lemma}
\newtheorem{claim}[theorem]{Claim}

\theoremstyle{definition}
\newtheorem{definition}[theorem]{Definition}

\theoremstyle{remark}
\newtheorem{remark}[theorem]{Remark}

\newcommand{\oldqed}{}
\def\endofFact{\hfill\scalebox{.6}{$\Box$}}
\newenvironment{claimproof}[1][Proof]{
  \renewcommand{\oldqed}{\qedsymbol}
  \renewcommand{\qedsymbol}{\endofFact}
  \begin{proof}[#1]
}{
  \end{proof}
  \renewcommand{\qedsymbol}{\oldqed}
} 

\title[Regularity inheritance in pseudorandom graphs (\today)]{Regularity inheritance in pseudorandom graphs}

  \author[\hfill P. Allen]{Peter Allen*}
  \author[J. B\"ottcher]{Julia B\"ottcher*}
  \author[J. Skokan]{Jozef Skokan*\dag}

  \thanks{
    *
    Department of Mathematics, London School of Economics, Houghton Street,
    London WC2A 2AE, U.K.
    \EMAIL{\{p.d.allen|j.boettcher|j.skokan\}@lse.ac.uk}}
 \thanks{
\dag
Department of Mathematics, University of Illinois at Urbana-Champaign, 1409 W. Green Street, Urbana, IL 61801, USA}
  \author[M. Stein\hfill\today]{Maya Stein\ddag}
  \thanks{
    \ddag
  Departamento de Ingenier\'ia Matem\'atica, Universidad de Chile,  Beauchef 851, Santiago, Chile
    \EMAIL{mstein@dim.uchile.cl}}
 \thanks{PA was partially supported by the EPSRC, grant number EP/P032125/1.}
 \thanks{JB was partially supported by the EPSRC, grant number EP/R00532X/1.}
 \thanks{JS was partially supported by the National Science Foundation, grant number DMS-1500121.}
 \thanks{MS is also affiliated to Centro de Modelamiento Matem\'atico, Universidad de Chile, UMI 2807 CNRS. MS acknowledges support by CONICYT + PIA/Apoyo a centros cient\'ificos y tecnol\'ogicos de excelencia con financiamiento Basal, C\'odigo AFB170001, by Millenium Nucleus Information and Coordination in Networks and by Fondecyt Regular Grant 1183080.}
\date{\today}

\begin{document}
\begin{abstract}
  Advancing the sparse regularity method, we prove one-sided and two-sided
  regularity inheritance lemmas for subgraphs of bijumbled graphs,
  improving on results of Conlon, Fox and Zhao [Adv. Math. 256 (2014),
  206--290]. These inheritance lemmas also imply improved $H$-counting lemmas
  for subgraphs of bijumbled graphs, for some~$H$.
\end{abstract}

\maketitle

\thispagestyle{empty}
\section{Introduction}

Over the past 40 years, the Regularity Method has developed into a powerful
tool in discrete mathematics, with applications in combinatorial geometry,
additive number theory and theoretical computer science (see
\cite{gerke05:_survey, kohayakawa01, komlos02, rodl10} for surveys). 

The Regularity Method relies on Szemer\'edi's celebrated Regularity
Lemma~\cite{szemeredi76} and a corresponding Counting Lemma.  
Roughly speaking, the Regularity Lemma states that each graph can 
(almost) be partitioned into a bounded number of regular pairs. More 
precisely, a pair $(U,W)$ of disjoint sets of vertices in a graph $G$ is 
\emph{$\eps$-regular} if, for all $U'\subset U$ and $W'\subset W$ 
with $|U'|\ge\eps|U|$ and $|W'|\ge\eps|W|$, we have $|d(U',W')-d(U,W)|\le \eps$,
where $d(U,W):=e(U,W)/(|U||W|)$ is the \emph{density} of the pair $(U,W)$
and $e(U,W)$ is the number of edges between $U$ and $W$ in $G$. 
The Regularity Lemma then says that every graph $G$ has a vertex partition 
$V_1\dcup\ldots\dcup V_{m}$ into almost equal-sized sets such that all but at most $\eps m^2$ pairs 
$(V_i,V_j)$ are $\eps$-regular and $m$ is bounded by a~function depending 
on $\eps$ but not on $G$.

The Counting Lemma complements the Regularity Lemma and states that in
systems of regular pairs the number of copies of any fixed graph~$H$ is
roughly as predicted by the densities of the regular pairs. In particular, if $H$ is a 
graph with vertex set $V(H)=[m]:=\{1,\dots,m\}$ and $G$ is an $m$-partite 
graph with partition $V_1\dcup\ldots\dcup V_{m}$ of $V(G)$ such that 
$(V_i,V_j)$ is $\eps$-regular whenever $ij\in E(H)$, then the number 
of (labelled) copies of $H$ in $G$ with vertex $i$ in $V_i$ for each $i\in V(H)$ is
$\prod_{ij\in E(H)}\big(d(V_i,V_j)\pm\gamma\big)\cdot\prod_{i\in[m]}|V_i|$,
as long as $\eps$ is sufficiently small.

Such a Counting Lemma can easily be proved with the help of the fact that
neighbourhoods in dense regular pairs are large and therefore `inherit'
regularity. More precisely, if $(X,Y)$, $(Y,Z)$ and $(X,Z)$ are
$\eps$-regular and have density $d\gg\eps$ then for most vertices $x\in X$
it is true that $|N(x)\cap Y|=(d\pm\eps)|Y|$ and $|N(x)\cap
Z|=(d\pm\eps)|Z|$. Hence one can easily deduct from $\eps$-regularity that
the pair $\big(N(x)\cap Y,Z\big)$ is $\eps'$-regular (this is called
\emph{one-sided inheritance}) and the pair $\big(N(x)\cap Y, N(x)\cap
Z\big)$ is $\eps'$-regular (this is called \emph{two-sided inheritance})
for some $\eps'$. Using this regularity inheritance, the Counting Lemma
follows by induction on the number of vertices~$m$ of~$H$.

\medskip

For sparse graphs~$G$, that is, $G$ with $n$ vertices and $o(n^2)$ edges,
the error term in the definition of $\eps$-regularity is too coarse, and
hence the Regularity Method is, as such, not useful for such graphs.
There are, however, sparse analogues of the Regularity Lemma, which
`rescale' the error term and hence are meaningful for sparse graphs.

\begin{definition}[sparse regularity]\label{def:sparsereg}
  Let $p>0$ and~$G$ be a graph. Let $U,W\subset V(G)$ be disjoint. 
  The \emph{$p$-density} of $(U,W)$ is $d_p(U,W):=e(U,W)/(p|U||W|)$.
  The pair $(U,W)$ is \emph{$(\eps,p)$-regular} if, for all $U'\subset U$
  and $W'\subset W$ with $|U'|\ge\eps|U|$ and $|W'|\ge\eps|W|$, we have
  \begin{equation*}
    \big|d_p(U',W')-d_p(U,W)\big|\le \eps\,.
  \end{equation*}
  It is $(\eps,d,p)$-regular if, moreover, $d_p(U,W)\ge d-\eps$.  An
  $(\eps)$-regular pair $(U,W)$ is an $(\eps,p)$-regular pair with density
  $d(U,W)=p$.
\end{definition}

The Sparse Regularity Lemma (see~\cite{Kohayakawa97Szemeredi,scott}) states
that the vertex set of any graph can be partitioned into sets, most pairs
of which are $(\eps)$-regular. However, a corresponding Counting Lemma for
$(\eps)$-regular pairs is not true in general: One can construct, say,
balanced $4$-partite graphs such that every pair of parts induces an
$(\eps,d,p)$-regular pair with $\eps\ll d$, but which do not contain a
single copy of $K_4$ (see, e.g., \cite[p.11]{CFZ14}).

Nevertheless, Counting Lemmas are known for sparse graphs~$G$ with additional
structural properties. In the case that~$G$ is a subgraph of a random
graph establishing such a Counting Lemma was a famous open problem,
the so-called K{\L}R-Conjecture~\cite{klr_conj}, which was settled only
  recently~\cite{BalMorSam,ConGowSamSch,SaxTho}. 
Proving an analogous result for subgraphs~$G$ of pseudorandom graphs has been another central
problem in the area. The study of pseudorandom graphs was initiated by
Thomason~\cite{Tho_pseudo1,Tho_pseudo2} (see also~\cite{krivelevich:_pseud} for more
background information on pseudorandom graphs), who considered a notion of
pseudorandomness very closely related to that of bijumbledness.

\begin{definition}[bijumbled]
 A pair $(U, V)$ of disjoint sets of vertices in a graph~$\Gamma$ is called
  \emph{$(p,\gamma)$-bijumbled in $\Gamma$} if, for all pairs $(U', V')$ with $U'\subset U$ and $V'\subset V$,  we have
  \begin{equation*}
    \big|e(U',V')-p|U'||V'|\big|\le \gamma \sqrt{|U'||V'|} \,.
 \end{equation*}
 A graph $\Gamma$ is said to be $(p,\gamma)$-bijumbled if all pairs of
 disjoint sets of vertices in~$\Gamma$ are $(p,\gamma)$-bijumbled in
 $\Gamma$. A bipartite graph $\Gamma$ with partition classes $U$ and $V$ is
 $(p,\gamma)$-bijumbled if the pair $(U, V)$ is $(p,\gamma)$-bijumbled in
 $\Gamma$.
\end{definition}

After partial results were obtained in~\cite{KoRoSchSiSk}, Conlon, Fox and
Zhao~\cite{CFZ14} recently proved a general Counting Lemma for subgraphs of
bijumbled graphs. This Counting Lemma has various interesting applications for subgraphs of bijumbled graphs, including a
Removal Lemma, Tur\'an-type results and Ramsey-type results. 

\medskip

For obtaining Counting Lemmas for sparse graphs the most straightforward
approach is to try to mimic the strategy for the proof of the dense Counting
Lemma outlined above.  The main obstacle here is that in sparse graphs it
is no longer true that neighbourhoods of vertices in regular pairs are
typically large and therefore trivially induce regular pairs -- they are of
size $pn\ll\eps n$. One can overcome this difficulty by establishing that,
under certain conditions, typically these sparse neighbourhoods nevertheless 
inherit sparse regularity. Inheritance Lemmas of this type were first
considered by Gerke, Kohayakawa, R\"odl, and Steger~\cite{GerKohRodSte}.
Conlon, Fox and Zhao~\cite{CFZ14} proved Inheritance Lemmas for subgraphs
of bijumbled graphs. The main results of the present paper are Inheritance Lemmas
which require weaker bijumbledness conditions. The first result establishes
one-sided regularity inheritance.

\begin{lemma}[One-sided Inheritance Lemma]\label{lem:oneside}
  For each $\eps',d>0$ there are $\eps,c>0$ such that for all $0<p<1$ the
  following holds.
  Let $G\subset \Gamma$ be graphs and $X,Y,Z$ be disjoint vertex sets in
  $V(\Gamma)$.  Assume that 
  \begin{itemize}
  \item $(X,Y)$ is $(p,cp^{3/2}\sqrt{|X||Y|})$-bijumbled
    in $\Gamma$,
  \item $(Y,Z)$ is $\big(p,cp^2(\log_2\tfrac{1}{p})^{-1/2}\sqrt{|Y||Z|}\big)$-bijumbled in $\Gamma$, and
   \item $(Y,Z)$ is $(\eps,d,p)$-regular in~$G$.  
  \end{itemize}
  Then, for all but at most at most $\eps'|X|$ vertices~$x$ of~$X$, the pair
  $\big(N_\Gamma(x)\cap Y,Z\big)$ is $(\eps',d,p)$-regular in~$G$.
\end{lemma}

Comparing this result with the analogue by Conlon, Fox,
Zhao in~\cite[Proposition~5.1]{CFZ14}, we need $\Gamma$ to be a factor
$(p\log_2\tfrac{1}{p})^{1/2}$ less jumbled when $|X|=|Y|=|Z|$.
The second result establishes two-sided regularity inheritance under
somewhat stronger bijumbledness conditions.

\begin{lemma}[Two-sided Inheritance Lemma]\label{lem:twoside}
  For each $\eps',d>0$ there are $\eps,c>0$ such that for all $0<p<1$ the
  following holds.  Let $G\subset \Gamma$ be graphs and $X,Y,Z$ be disjoint
  vertex sets in  $V(\Gamma)$.  Assume that
  \begin{itemize}
  \item $(X,Y)$ is $(p,cp^{2}\sqrt{|X||Y|})$-bijumbled
    in $\Gamma$,
  \item $(X,Z)$ is $(p,cp^3\sqrt{|X||Z|})$-bijumbled in $\Gamma$,
  \item $(Y,Z)$ is $(p,cp^{5/2}\big(\log_2\tfrac{1}{p}\big)^{-\frac12}\sqrt{|Y||Z|})$-bijumbled in $\Gamma$, and
  \item $(Y,Z)$ is $(\eps,d,p)$-regular in~$G$.  
  \end{itemize}
  Then, for all but at most $\eps'|X|$ vertices~$x$ of~$X$, the pair
  $\big(N_\Gamma(x)\cap Y,N_\Gamma(x)\cap Z\big)$ is $(\eps',d,p)$-regular in~$G$.
\end{lemma}

Here $\Gamma$ needs to be a factor $p$
less jumbled when $|X|=|Y|=|Z|$ than in~\cite[Proposition~1.13]{CFZ14}.
We remark that the bijumbledness conditions in our results imply that these
implicitly are statements about sufficiently large graphs (see Lemma~\ref{lem:n0}).
Our proofs use the counting lemma for~$C_4$ of Conlon, Fox,
Zhao~\cite{CFZ14} as a fundamental ingredient.

\subsection{Applications}

\mbox{}
\medskip

\paragraph{\bf Blow-up Lemmas}
Blow-up Lemmas are an important tool in the Regularity
Method, which make it possible to derive results about large or even
spanning subgraphs in certain graph classes (see, e.g., \cite{KueOst_embeddingSurvey}).
In~\cite{blowup} a Blow-up Lemma which works relative to sparse
jumbled graphs is proved. The proof of this lemma relies on our Regularity
Inheritance Lemmas,
Lemmas~\ref{lem:oneside} and~\ref{lem:twoside}.

\medskip

\paragraph{\bf Resilience theorems in jumbled graphs}

As an application of the Blow-up Lemma for jumbled graphs in~\cite{ABET} \emph{resilience
problems} for jumbled graphs with respect to certain spanning subgraphs are
considered. The study of such problems dates back
to~\cite{millenium} where the name \emph{fault-tolerance}
was used, but lately the term
\emph{resilience} has come into vogue,
following Sudakov and Vu~\cite{SudVu}.

In~\cite{ABET} Lemmas~\ref{lem:oneside}
and~\ref{lem:twoside} together with the Blow-up Lemma for jumbled graphs
are used to derive the following sparse version of the Bandwidth Theorem (proved for
dense graphs in~\cite{BolKom}).
\begin{theorem} \cite{ABET}\ %
 For each $\eps >0$, $\Delta \geq 2$, and $k \geq 1$, there exists a constant $c >0$ such that the following holds for any $p>0$. Given $\gamma\le cp^{\max(4,(3\Delta+1)/2)}n$, suppose $\Gamma$ is a $\big(p,\gamma\big)$-bijumbled graph, $G$ is a spanning subgraph of $\Gamma$ with $\delta(G) \geq\big(\tfrac{k-1}{k}+\eps\big)pn$, and $H$ is a $k$-colourable graph on $n$ vertices with $\Delta(H) \leq \Delta$ and bandwidth at most $c n$. Suppose further that there are at least $c^{-1}p^{-6} \gamma^2n^{-1}$ vertices in $V(H)$ that are not contained in any triangles of $H$. Then $G$ contains a copy of $H$. 
\end{theorem}

Note that the bijumbledness requirement implicitly places a lower bound on $p$. It is necessary to insist on some vertices of $H$ not being in any triangles of $H$, but the number $c^{-1}p^{-6}\gamma^2n^{-1}$ comes from the requirements of Lemma~\ref{lem:twoside}, and improvement there would immediately improve this statement\footnote{As we discuss in Section~\ref{subsec:opt} below, we believe that one can improve Lemma~\ref{lem:twoside} in order to obtain $c^{-1}p^{-4}\gamma^2n^{-1}$ uncovered vertices. This is the best achievable using inheritance lemmas, but we are not sure whether the constructions giving a lower bound on inheritance lemmas can be modified to give a matching lower bound in this setting.} . This is a very general resilience result, covering for example Hamilton cycles, clique factors, and much more. Note that although a Hamilton cycle might not be $2$-colourable, in~\cite{ABET} a more complicated variant of the above statement is proved which allows occasional vertices to receive a $(k+1)$\textsuperscript{st} colour.

%

\medskip

\paragraph{\bf Counting Lemmas}
The most obvious application of our inheritance lemmas is to prove stronger
Counting Lemmas than those in~\cite{CFZ14}. The results we obtain are not
much stronger than those in~\cite{CFZ14}, so we do not regard this as a
main contribution of this paper. However we feel it is worth providing the
stronger results for future use, and that the (rather different to that in~\cite{CFZ14})
approach we take is worth highlighting.

Recall that for a dense graph
$G$ and fixed $H$ the Counting Lemma provides matching upper and lower bounds on the
number of copies of $H$ in $G$. By contrast, when $G$ is a subgraph of a
sparse bijumbled graph $\Gamma$, we formulate two separate Counting
Lemmas. The one-sided Counting Lemma gives only a lower bound on the number
of copies of $H$ in $G$, while the two-sided Counting Lemma gives in
addition a matching upper bound.\footnote{Somewhat confusingly, the terms
  one-sided/two-sided refer to completely different aspects in the
  one-sided/two-sided Counting Lemmas and the one-sided/two-sided
  Inheritance Lemmas. Both are standard terminology.}  The motivation for
formulating two separate lemmas is that for many graphs $H$, the
bijumbledness requirement on $\Gamma$ to prove a one-sided Counting Lemma
is significantly less than to prove a two-sided Counting Lemma, and for
many applications the one-sided Counting Lemma suffices. 

The statements and proofs of our Counting Lemmas are quite technical, and
we prefer to leave them as an Appendix to this paper. Comparison with the
results of~\cite{CFZ14} is unfortunately also not straightforward, in part because the
two-sided Counting Lemma in~\cite{CFZ14} actually provides better performance than the
one-sided Counting Lemma there in some important cases, such as for
cliques. Briefly, our one-sided Counting Lemma always performs at least as
well as either of~\cite[Theorems 1.12 and 1.14]{CFZ14}, and in some cases
our results are better. For example, if $H$ consists of $10$ copies of
$K_3$ sharing a single vertex, then our one-sided Counting Lemma requires
$(p,cp^3)$-jumbledness to lower bound the number of copies of $H$, whereas
the results in~\cite{CFZ14} require $(p,cp^4)$-bijumbledness. Our two-sided
Counting Lemma sometimes performs better than~\cite[Theorem
1.12]{CFZ14}. Again, for $10$ copies of $K_3$ sharing a vertex, we require
$\big(p,cp^{10.5}\big)$-bijumbledness while~\cite{CFZ14} requires
$\big(p,cp^{12}\big)$-bijumbledness. In general, our results perform better
when there are vertices of exceptionally high degree. 
For many interesting graphs (such as $d$-regular graphs for any
$d\ge 3$) the performance is identical.

Of course, these counting lemmas can also be immediately applied in the
(relatively straightforward) applications presented in~\cite{CFZ14}.
For most of these applications what one requires
is a one-sided Counting Lemma. In particular, 
by using the one-sided Counting Lemma resulting from our
Inheritance Lemmas the bijumbledness requirements
for the removal lemma~\cite[Theorem~1.1]{CFZ14}, the Tur\'an
result~\cite[Theorem~1.4]{CFZ14}, and the Ramsey
result~\cite[Theorem~1.6]{CFZ14} can always be matched, and in some cases
be improved.

\subsection{Optimality}\label{subsec:opt}

\mbox{}
\medskip

Our one-sided Inheritance Lemma is probably not optimal.
In the case when~$H$ is a clique, Conlon, Fox and
Zhao~\cite{CFZ14} are able to obtain a one-sided counting lemma with a
bijumbledness requirement matching ours by using a completely different strategy.
In particular, when~$H$ is a triangle, these counting lemmas imply a
triangle removal lemma for subgraphs of bijumbled graphs with
$\beta=o(p^3n)$. Such a result was obtained earlier already
in~\cite{KRoSchSk05}, where it was also conjectured that this can be
improved to $\beta=o(p^2n)$. Conlon, Fox and Zhao~\cite{CFZ14} conjecture
the contrary. We sympathise with the former conjecture, and believe that it
would be extremely interesting to resolve this question. We think it
unlikely that Lemma~\ref{lem:twoside} is optimal and believe there is room for
improvement in our proof strategy; any improvement would disprove the conjecture of Conlon, Fox and Zhao.

\subsection*{Organisation}
The remaining sections of this paper are devoted to the proofs of the Inheritance Lemmas. We
start in Section~\ref{sec:overview} with an overview of these proofs.
Section~\ref{sec:prelim} collects necessary auxiliary results on bijumbled
graphs and sparse regular pairs. In Sections~\ref{sec:C4} and~\ref{sec:tbh}
we prove various lemmas used in the proofs of the Inheritance Lemmas:
Section~\ref{sec:C4} establishes lemmas on counting copies of $C_4$ in various
bipartite graphs, and Section~\ref{sec:tbh} concerns a classification of
pairs of vertices in such graphs according to their codegrees. In
Section~\ref{sec:oneside} we prove Lemma~\ref{lem:oneside} and in
Section~\ref{sec:twoside} Lemma~\ref{lem:twoside}.

\subsection*{Notation}

For a graph $G=(V,E)$ we also write $V(G)$ for the vertex set and $E(G)$
for the edge set of~$G$. We write $e(G)$ for the number of edges
of~$G$. For vertices $v,v'\in V$ and a set $U\subset V$ we write $N_G(v;U)$
and $N_G(v,v';U)$ for the $G$-neighbourhood of~$v$ in~$U$ and common
$G$-neighbourhood of~$v$ and~$v'$ in~$U$, respectively. Similarly,
$\deg_G(v;U):=|N_G(v;U)|$ and $\deg_G(v,v';U):=|N_G(v,v';U)|$. If $U=V$ we
may omit~$U$ and, if~$G$ is clear from the context, we may also omit~$G$.

For disjoint vertex sets $U,W\subset V$ the graph $G[U,W]$ is the bipartite
subgraph of~$G$ containing exactly all edges of~$G$ with one end in~$U$ and
the other in~$W$. We write $e(U,W)$ for the number of edges in $G[U,W]$.

\section{Proof Overview}
\label{sec:overview}

We sketch the proof of Lemma~\ref{lem:oneside} first.
We label the pairs in $Y$ as `typical', `heavy', or `bad', according to whether their $G$-common neighbourhood in $Z$ is not significantly larger than one would expect, or so large as to be unexpected even in $\Gamma$, or intermediate. By using the bijumbledness of $(Y,Z)$ in $\Gamma$ we can show that the heavy pairs are so few that one can ignore them (Lemma~\ref{lem:C4heavy}).

Now suppose that $x\in X$ is such that $\big(N_\Gamma(x;Y),Z\big)$ is either too dense or is not sufficiently regular. In either case, by several applications of the defect Cauchy-Schwarz inequality, we conclude that $\big(N_\Gamma(x;Y),Z\big)$ contains noticeably more copies of $C_4$ in $G$ than one would expect if $(Y,Z)$ were a random bipartite graph of the same density (Lemma~\ref{lem:C4}). In particular, the average pair of vertices in $N_\Gamma(x;Y)$ has noticeably more $G$-common neighbours in $Z$ than one would expect. It follows that a substantial fraction of the pairs $y,y'$ in $N_\Gamma(x;Y)$ are bad or heavy. Since there are few heavy pairs, we see that there are many bad pairs (Lemma~\ref{lem:badpairs}).

On the other hand, because $(Y,Z)$ is regular, we can count copies of $C_4$
in $G$ crossing the pair (Lemma~\ref{lem:c4count}, which is taken from~\cite{CFZ14}). A further application of the defect Cauchy-Schwarz inequality tells us that a very small fraction of the pairs in $Y$ are bad, and using the bijumbledness of $(X,Y)$ we conclude that there are few triples $(x,y,y')$ such that $xy$ and $xy'$ are edges of $\Gamma$ and $(y,y')$ is bad (Lemma~\ref{lem:fewbad}).

Putting these two statements together, we conclude that there are few $x\in X$ such that $\big(N_\Gamma(x;Y),Z\big)$ is either too dense or is not sufficiently regular. By averaging, if there are few dense pairs there are also few pairs which are too sparse. This completes the proof of Lemma~\ref{lem:oneside}.

The proof of Lemma~\ref{lem:twoside} is very similar. We have to additionally classify the pairs in $Y$ as typical, heavy or bad with respect to $x\in X$, which we do according to their $G$-common neighbourhood in $N_\Gamma(x;Z)$. Now Lemma~\ref{lem:badpairs} as before tells us that if $x\in X$ is such that $\big(N_\Gamma(x;Y),Z\big)$ is either too dense or is not sufficiently regular, then a substantial fraction of the pairs $(y,y')$ in $N_\Gamma(x;Y)$ are bad with respect to $x$. Lemma~\ref{lem:fewbad} continues to tell us that there are few triples $(x,y,y')$ such that $xy$ and $xy'$ are edges of $\Gamma$ and $(y,y')$ is bad, and Lemma~\ref{lem:C4heavy} continues to tell us that we can ignore the heavy pairs. To complete the argument as before, it remains to show that if $(y,y')$ is a typical pair, then there are few $x$ such that $xy,xy'\in\Gamma$ and $(y,y')$ is bad with respect to $x$. To prove this we do not use the requirement $xy,xy'\in\Gamma$, but simply bound, using bijumbledness of $(X,Z)$, the number of $x$ with abnormally many neighbours in $N_G(y,y';Z)$. This step is where we require most bijumbledness. We believe it is wasteful, but were not able to find a more efficient way.

\section{Preliminaries}
\label{sec:prelim}

\subsection{Bijumbledness}

One consequence of a pair $(U, V)$ being $(p,\gamma)$-bijumbled is that  most vertices in $U$ have about $p|V|$ neighbours in $V$.

\begin{lemma}\label{lem:degbound}
  Let $k\ge 1$, $c'>0$, and $0<p<1$, and let $(U, V)$ be a $(p,c'p^k\sqrt{|U||V|})$-bijumbled pair in a~graph $\Gamma$. Then, for any  $\gamma>0$, we have
  $$
    \big|\{u\in U\colon \deg_\Gamma(u;V)\neq(1\pm\gamma)p|V|\}\big|
      \le 
    2(c')^2p^{2k-2}\gamma^{-2}|U|\,.
  $$
\end{lemma}

\begin{proof}
  Let $U^+:=\{u\in U\colon \deg_\Gamma(u;V)>(1+\gamma)p|V|\}$. 
  By bijumbledness applied to the pair $(U^+,V)$ we have
  $$
   (1+\gamma)p|U^+||V| < e(U^+,V) 
    \le
   p|U^+||V|+c'p^k\sqrt{|U||V|}\sqrt{|U^+||V|}\,.
 $$ 
 Simplifying this gives $|U^+|\le (c')^2p^{2k-2}\gamma^{-2}|U|$. A similar 
 calculation for the set $U^-$ of vertices in $U$ with fewer than $(1-\gamma)p|V|$
 neighbours in $V$ yields the same bound on $|U^-|$, and the result follows.
\end{proof}

Moreover, non-trivial bijumbled graphs cannot be very small.

\begin{lemma}\label{lem:n0}
 Let $0<c'\le \tfrac14$,  $0<p\le\tfrac14$ and $k\ge 1$. Let~$\Gamma$ be a graph, and let $(U,V)$ be $(p,c'p^k\sqrt{|U||V|})$-bijumbled in~$\Gamma$. Then we have
 \[|U|,|V|\ge\tfrac18 (c')^{-2}p^{1-2k}\,.\]
\end{lemma}
\begin{proof}
  By Lemma~\ref{lem:degbound}, the number of vertices in $U$ with more than
  $2p|V|$ neighbours in $V$ is at most
  $2(c')^2p^{2k-2}|U|\le\tfrac12|U|$. It follows that we can take a set
  $U'\subset U$ of $\min\big\{\tfrac{1}{4}p^{-1},\tfrac12|U|\big\}$ vertices,
  each with degree at most $2p|V|$. The union of their neighbourhoods covers by
  definition at most $\tfrac{1}{4}p^{-1} \cdot 2p|V|=\tfrac12|V|$ vertices of $V$, 
  so we can let $V'$ be a subset of $\tfrac12|V|$ vertices in $V$ with no edges 
  between $U'$ and $V'$. Applying bijumbledness to the pair $(U',V')$, we have
  $$
    0 = e\big(U',V'\big) \ge p|U'||V'| - c'p^k\sqrt{|U||V|}\sqrt{|U'||V'|}\,,
  $$
  which implies $(c')^2p^{2k}|U||V|\ge
  p^2|U'||V'|=p^2\min\big\{\tfrac{1}{4}p^{-1},\tfrac12|U|\big\}\frac12|V|$.
  Hence, we obtain
  $$
   |U|
    \ge 
   \tfrac12(c')^{-2}p^{2-2k}\min\big\{\tfrac14p^{-1},\tfrac12|U|\big\}\,.
  $$
  The inequality $|U|\ge\tfrac14(c')^{-2}p^{2-2k}|U|$ is false for all
  $U\neq\emptyset$ by our choice of $c'$, $p$ and $k$, so we conclude that
  $$
   |U|\ge \tfrac18(c')^{-2}p^{1-2k}\,.
  $$
  The same bound applies to $|V|$.
\end{proof}
\begin{remark}
  Erd\H{o}s and Spencer~\cite{ErSp1972} (see also Theorem~5 
  in~\cite{ErGoPaSp}) observed that there exists $c>0$ such that every
  $m$-vertex graph with density $p$ contains two disjoint sets $X$ and $Y$ for
  which $\big|e(X,Y)-p|X||Y| \big| \geq c\sqrt{pm} \sqrt{|X||Y|}$, as long as
  $p(1-p)\geq 1/m$. One can also recover Lemma \ref{lem:n0} using this result. 
  (See also Remark 6 in \cite{KoRoSchSiSk}.)
\end{remark}

\subsection{Sparse regularity}
\label{sec:reg}

The Slicing Lemma, Lemma~\ref{lem:slicing}, states that large subpairs of
regular pairs remain regular. Its proof, which we omit, follows directly from Definition~\ref{def:sparsereg}.

\begin{lemma}[Slicing Lemma]\label{lem:slicing}
  For any $0<\eps<\gamma$ and any $p>0$, any $(\eps,p)$-regular pair $(U,W)$ in $G$, and any $U'\subset U$ and $W'\subset W$ with $|U'|\ge\gamma|U|$ and $|W'|\ge\gamma|W|$, the pair $(U',W')$ is $(\eps/\gamma,p)$-regular in $G$ with $p$-density $d(U,W)\pm\eps$.
\end{lemma}

In the other direction, the following lemma shows that, under certain conditions, adding a few vertices to either side of a regular pair cannot destroy regularity
completely.


\begin{lemma}\label{lem:sticking}
  Let $0<\eps<\tfrac{1}{10}$ and $c\le\tfrac{1}{10}\eps^3$. Let $G$ be  
  a~spanning subgraph of a graph $\Gamma$, let $(U',V')$ be a pair of disjoint sets 
  in~$V(\Gamma)$,  and let $U\subset U'$ and $V\subset V'$. 
  Assume $(U',V')$ is $(p,cp\sqrt{|U||V|})$-bijumbled in $\Gamma$ and  
  $(U,V)$ is $(\eps,d,p)$-regular in~$G$. 

   If $|U'|\le\big(1+\tfrac{1}{10}\eps^3\big)|U|$ and  
  $|V'|\le\big(1+\tfrac{1}{10}\eps^3\big)|V|$, then $(U',V')$ is
  $(2\eps,d,p)$-regular in~$G$.
\end{lemma}

\begin{proof}
  Let $X\subseteq U'$ with $|X|\ge 2\eps|U'|$ and $Y\subseteq V'$ with
  $|Y|\ge 2\eps|V'|$ be arbitrary.
  Using $(p,cp\sqrt{|U||V|})$-bijumbledness of $(U',V')$ in~$\Gamma$ we have
  \begin{align*}
   e(X\cap U,Y\setminus V)&\le e_\Gamma(X\cap U,Y\setminus V) \le p|U|\cdot \tfrac{\eps^3}{10}|V|+cp\sqrt{|U||V|}\sqrt{|U|\cdot \tfrac{\eps^3}{10}|V|}\\
&\le \tfrac{1}{5}\eps^3p|U||V|\,.
\end{align*}
  Similarly, we have
  \begin{align*}
    e(X\setminus U,Y)
    &\le p\tfrac{\eps^3}{10}|U|\cdot \big(1+\tfrac{\eps^3}{10}\big)|V|+cp\sqrt{|U||V|}\sqrt{\tfrac{\eps^3}{10}|U|\cdot \big(1+\tfrac{\eps^3}{10}\big)|V|}\\
    &\le\tfrac15\eps^3p|U||V|\,.
  \end{align*}
  Moreover, since $(U,V)$ is $(\eps,d,p)$-regular,
  $e(X\cap U,Y\cap V)=(d\pm\eps)p|X\cap U||Y\cap V|$.
  Hence
 \begin{equation*}\begin{split}
   e(X,Y)&=e(X\cap U,Y\cap V)+e(X\cap U,Y\setminus V)+e(X\setminus U,Y) \\
   &=(d\pm\eps)p|X\cap U||Y\cap V|\pm\tfrac{2}{5}\eps^3p|U||V|\\
   &=\big(d\pm\tfrac{3}{2}\eps\big)p|X\cap U||Y\cap V|\\ &
   =(d\pm 2\eps)p|X||Y|\,.
  \end{split}\end{equation*}
  We conclude that $(U',V')$ is $(2\eps,d,p)$-regular in $G$.
\end{proof}

\subsection{Cauchy-Schwarz}

We use the following `defect' form of the Cauchy-Schwarz inequality. This inequality and a proof can be found in~\cite[Fact B]{FR02}. 

\begin{lemma}[Defect form of Cauchy-Schwarz]\label{lem:CSdefect}
  Let $a_1,\ldots,a_k$ be real numbers with average at least~$a$. If for some
  $\delta\ge 0$ at least $\mu k$ of them average at least $(1+\delta)a$, then
  \[\sum_{i=1}^ka_i^2\ge ka^2\big(1+\tfrac{\mu\delta^2}{1-\mu}\big)\,,\]
  and the same bound is obtained if at least $\mu k$ of the $a_i$ average at
  most $(1-\delta)a$.
\end{lemma}

\section{Counting copies of \texorpdfstring{$C_4$}{C4} in regular, irregular and dense pairs}
\label{sec:C4}

The following counting lemma for counting~$C_4$ in $(\eps,d,p)$-regular
subgraphs of bijumbled graphs is as given by Conlon, Fox, and Zhao \cite[Proposition 4.13]{CFZ14}.
We write $C_4(G)$ for the number of unlabelled copies of~$C_4$ in~$G$.

\begin{lemma}[counting $C_4$ in regular pairs]
  \label{lem:c4count} 
  For any $\eps>0$, $c>0$ and $d\in[0,1]$  the following holds. 
  If $(U, V)$ is a $(p,cp^2\sqrt{|U||V|})$-bijumbled pair in a graph $\Gamma$,
  and $G$ is a bipartite subgraph of $\Gamma$ with parts $U$ and $V$ 
  which forms an $(\eps,d,p)$-regular pair, then 
  $C_4(G)=\tfrac14\big(d^4\pm 100(c+\eps)^{1/2}\big)p^4|U|^2|V|^2$.
\end{lemma}

The next lemma gives a lower bound on the number of copies of $C_4$ in a
bipartite graph of a given density. Moreover, if this bipartite graph is
not $(\eps)$-regular we obtain an even stronger lower bound.
Observe that for this lemma we do
\emph{not} require that the pair is a subgraph of a pseudorandom graph.

\begin{lemma}[counting $C_4$ in dense pairs and irregular pairs]\label{lem:C4}
  Let $0<\eps_{\sublem{lem:C4}}\le 10^{-3}$, let $G$ be a bipartite graph 
  with vertex classes $U$ and $V$ of sizes $m\ge n\ge 2\eps_{\sublem{lem:C4}}^{-9}$ respectively. Suppose that 
  $G$ has density $q\ge\eps_{\sublem{lem:C4}}^{-10}n^{-1/2}$. 
  \begin{enumerate}[label=\abc]
  \item\label{itm:C4:dense}
    $C_4(G)\ge(1-\eps_{\sublem{lem:C4}}^8)q^4 \frac14m^2n^2$.
 \item\label{itm:C4:irreg}
    If~$G$ is not $(\eps_{\sublem{lem:C4}})$-regular, then we have $C_4(G)
   \ge (1+\eps_{\sublem{lem:C4}}^{13})q^4\frac14m^2n^2$.
  \end{enumerate}
\end{lemma} 
\begin{proof}
  Assume~$G$ has density~$q$. 
  Clearly, we have
  \begin{equation}\label{eq:regC4:C4_common}
    C_4(G) = \sum_{\{u,u'\}\in\binom{U}{2}}\binom{\deg(u,u')}{2} \,.
  \end{equation}
  Hence, for bounding this quantity we will analyse common
  neighbourhoods of vertices in~$U$. Let us first bound the average 
  $$a:=\binom{m}{2}^{-1}\sum_{u\neq u' \in U} \deg(u,u').$$ 
 Observe that if $a\ge (1+\eps_{\sublem{lem:C4}}^8)q^2n$ then, using Jensen's inequality and facts that $q\ge 2\eps_{\sublem{lem:C4}}^{-4}n^{-1/2}$ and $m\ge n\ge 2\eps_{\sublem{lem:C4}}^{-9}$, we get 
  \begin{align*}
    C_4(G)\ge \binom{m}{2}\binom{(1+\eps_{\sublem{lem:C4}}^8)q^2n}{2}
   & \ge \frac{(1+\eps_{\sublem{lem:C4}}^8)(1+\eps_{\sublem{lem:C4}}^9)q^4n^2}{2} \binom{m}{2} \\
  &  \ge (1+\eps_{\sublem{lem:C4}}^8)q^4\tfrac{1}{4}n^{2}m^{2}\,,
  \end{align*}
  and thus are done. Hence we may assume in the following that
  \begin{equation}\label{eq:regC4:ale}
    a\le(1+\eps_{\sublem{lem:C4}}^8)q^2 n \,.
  \end{equation}
  For obtaining a corresponding lower bound on~$a$ note that the average
  degree of the vertices in~$V$ is $qm$. Hence by Jensen's inequality we
  have
  \begin{equation*}
    \sum_{v\in V}\binom{\deg(v)}{2} 
    \ge n\binom{qm}{2} = n \frac{qm(qm-1)}{2} \ge n \frac{(1-\eps_{\sublem{lem:C4}}^{20}) q^2m^2}{2}\,,
  \end{equation*}
  where the second inequality uses $q\ge q^2\ge \eps_{\sublem{lem:C4}}^{-20}m^{-1}$.  Therefore
  \begin{equation*}
    \sum_{\{u,u'\}\in\binom{U}{2}} \deg(u,u') = 
    \sum_{v\in V}\binom{\deg(v)}{2}
    \ge n \frac{(1-\eps_{\sublem{lem:C4}}^{20}) q^2m^2}{2}
    \ge (1-\eps_{\sublem{lem:C4}}^{20})  q^2\binom{m}{2} n \,.
  \end{equation*}
  This gives
 \begin{equation}\label{eq:regC4:age}
    a \ge (1-\eps_{\sublem{lem:C4}}^{20}) q^2n.
  \end{equation}
  Moreover, we obtain from~\eqref{eq:regC4:C4_common} and~\eqref{eq:regC4:ale}
  that
 \begin{equation}\label{eq:regC4:C4}
    C_4(G) \ge  \frac{1}{2}\sum_{\{u,u'\}\in\binom{U}{2}}\deg(u,u')^2  - (1+\eps_{\sublem{lem:C4}}^8) n q^2\binom{m}{2}\,.
  \end{equation}
  For estimating the sum of squares in this inequality, we will use
  the defect form of Cauchy-Schwarz (Lemma~\ref{lem:CSdefect}).

  Let us first establish the first part of Lemma~\ref{lem:C4}.
  We apply Lemma~\ref{lem:CSdefect} with
  $k=\binom{m}{2}$, $\mu=\delta=0$ (so actually without defect) to obtain that
  \begin{equation*}
    \sum_{\{u,u'\}\in\binom{U}{2}}\deg(u,u')^2
    \ge \binom{m}{2} a^2 
    \geByRef{eq:regC4:age} \binom{m}{2} (1-\eps_{\sublem{lem:C4}}^{20})^2q^4n^2\,.
 \end{equation*}
  Hence, by~\eqref{eq:regC4:C4}, we have
  \begin{align*}
    C_4(G)& \ge
    \frac12\binom{m}{2}(1-\eps_{\sublem{lem:C4}}^{20})^2q^4n^2 
     -
    (1+\eps_{\sublem{lem:C4}}^8) nq^2\binom{m}{2} \\
   & \ge
   (1-2\eps_{\sublem{lem:C4}}^{20})q^4\frac{n^2}{2}\binom{m}{2}
   -
   n\binom{m}{2}q^4\big(\eps_{\sublem{lem:C4}}^{-10}n^{-1/2}\big)^{-2} \\
   & \ge
  (1-\eps_{\sublem{lem:C4}}^{8})q^4 \frac14m^2n^2\,,
  \end{align*}
  as desired, where we used $q\ge \eps_{\sublem{lem:C4}}^{-10}n^{-1/2}$ in the second inequality.

  For the second part of the lemma, we will use a similar calculation, but
  we will apply
  Lemma~\ref{lem:CSdefect} with $\mu,\delta>0$.
  So we need to find a subset $\tilde U \subset U$ of vertices
  whose average pair degrees differ significantly from~$a$.
  
  The following definition will be useful. For a set $\tilde U \subset U$, let \begin{equation}\label{eq:regC4:a'}
      a(\tilde U):=\binom{|\tilde U|}{2}^{-1} \sum_{\{u,u'\}\in\binom{\tilde U}{2}} \deg(u,u')
      = \binom{|\tilde U|}{2}^{-1} \sum_{v\in V}\binom{\deg(v,\tilde U)}{2}.
    \end{equation}
  
  \begin{claim}\label{clm:tildeU}
    If~$G$ is not $(\eps_{\sublem{lem:C4}})$-regular,
    then there is a set $\tilde U \subset U$ with 
    $|\tilde U|\ge\eps_{\sublem{lem:C4}} m$ such that 
    $$
      a(\tilde U)\ge (1+2\eps_{\sublem{lem:C4}}^5)q^2n 
                   \ge (1+\eps_{\sublem{lem:C4}}^5)a,
    $$ 
    where the second inequality follows from~\eqref{eq:regC4:ale}.
  \end{claim}
  Before we prove this claim, let us show how it implies the second part of our
  lemma. For this, assume that~$G$ is not $(\eps_{\sublem{lem:C4}})$-regular, and let $\tilde U$ be the set guaranteed by Claim \ref{clm:tildeU}.
  Since $|\tilde U|\ge\eps_{\sublem{lem:C4}} m$, there are at least
  $\binom{\eps_{\sublem{lem:C4}} m}{2}\ge \tfrac12\eps_{\sublem{lem:C4}}^2\binom{m}{2}$ pairs of vertices
  in~$\tilde U$. Thus we can use Lemma~\ref{lem:CSdefect} with
  $k:=\binom{m}{2}$, $\mu=\eps_{\sublem{lem:C4}}^2/2$ and $\delta=\eps_{\sublem{lem:C4}}^5$ to infer that
  \begin{align*}
    \sum_{\{u,u'\}\in\binom{U}{2}}\deg(u,u')^2 
   & \ge \binom{m}{2} a^2 \Big(1+\frac{\eps_{\sublem{lem:C4}}^{12}}{2}\Big) \\
   & \geByRef{eq:regC4:age} \binom{m}{2} (1-\eps_{\sublem{lem:C4}}^{20})^2q^4n^2 \Big(1+\frac{\eps_{\sublem{lem:C4}}^{12}}{2}\Big)\\
 &    \ge \binom{m}{2} \Big(1+\frac{\eps_{\sublem{lem:C4}}^{12}}{4}\Big)q^4n^2 \,.
  \end{align*}
  Together with~\eqref{eq:regC4:C4} this gives the desired
  \begin{align*}
    C_4(G) & \ge\frac12\binom{m}{2}\Big(1+\frac{\eps_{\sublem{lem:C4}}^{12}}{4}\Big)q^4n^2 -
    (1+\eps_{\sublem{lem:C4}}^8) nq^2\binom{m}{2}\\
   & \ge\frac12\binom{m}{2}\Big(1+\frac{\eps_{\sublem{lem:C4}}^{12}}{5}\Big)q^4n^2 \\
   & \ge(1+\eps_{\sublem{lem:C4}}^{13})q^4\tfrac14 n^2m^2\,,
  \end{align*}
  where again we used $q\ge \eps_{\sublem{lem:C4}}^{-10}n^{-1/2}$ in the second inequality.

  It remains to prove the claim.

  \begin{claimproof}[Proof of Claim~\ref{clm:tildeU}]
    Since~$G$ is not $(\eps_{\sublem{lem:C4}})$-regular there are sets 
    $U'\subset U$ and $V'\subset V$ with $|U'|=\eps_{\sublem{lem:C4}} m$ and 
    $|V'|= \eps_{\sublem{lem:C4}} n$ such that either
  \begin{equation}\label{icecream}
    \text{$d(U',V')>(1+\eps_{\sublem{lem:C4}})q$ \ or \ $d(U',V')<(1-\eps_{\sublem{lem:C4}})q$.}
    \end{equation}
      Now we distinguish three cases.

    First suppose that $d(U',V)\ge \big(1+\tfrac{\eps_{\sublem{lem:C4}}^3}{10}\big)q=:\tilde q$.  Then, using again Jensen's inequality, we have
    \begin{align*}
      a(U')
     & = \binom{\eps_{\sublem{lem:C4}} m}{2}^{-1} \sum_{v\in V}\binom{\deg(v,U')}{2}
     \\ & \ge \frac{2}{\eps_{\sublem{lem:C4}}^2 m^2} \cdot n\frac{\tilde q\eps_{\sublem{lem:C4}} m(\tilde q \eps_{\sublem{lem:C4}}
        m-1)}{2} \\
      & \ge \frac{2}{\eps_{\sublem{lem:C4}}^2 m^2} \cdot n \frac{(1-\eps_{\sublem{lem:C4}}^{7}) \tilde
        q^2\eps_{\sublem{lem:C4}}^2m^2}{2} 
      \\ & = (1-\eps_{\sublem{lem:C4}}^{7})\tilde q^2 n
      \ge \left(1+\frac{\eps_{\sublem{lem:C4}}^3}5\right) q^2 n\,,
    \end{align*}
    where the second inequality uses $\tilde q \ge q\ge \eps_{\sublem{lem:C4}}^{-8}/m$ and the last inequality uses $\eps_{\sublem{lem:C4}}\le 10^{-3}$. Hence we can choose~$U'$ as~$\tilde U$.

    Secondly, suppose that $d(U',V)\le \big(1-\tfrac{\eps_{\sublem{lem:C4}}^3}{10}\big)q$ and let $U'':=U\setminus U'$. Then 
    \begin{equation*}
      d(U'',V) = \frac{d(U,V)nm-d(U',V)\eps_{\sublem{lem:C4}} nm}{(1-\eps_{\sublem{lem:C4}})nm}
      \ge\frac{q-(1-\frac{\eps_{\sublem{lem:C4}}^3}{10})q\eps_{\sublem{lem:C4}}}{1-\eps_{\sublem{lem:C4}}}\ge \left(1+\frac{\eps_{\sublem{lem:C4}}^4}{10}\right)q \,.
    \end{equation*}
    Using an analogous calculation as in the previous case we obtain
    $a(U'')\ge(1+2\eps_{\sublem{lem:C4}}^5)q^2n$ and thus can choose~$U''$ as~$\tilde U$. 

    Finally, suppose $\big(1-\tfrac{\eps_{\sublem{lem:C4}}^3}{10}\big)q < d(U',V) < \big(1+\tfrac{\eps_{\sublem{lem:C4}}^3}{10}\big)q$. In this case we
    will use $\tilde U:=U'$ and apply Lemma~\ref{lem:CSdefect} to bound
   \begin{equation}\label{eq:regC4:aU'}
      a(U')\ge 
     \frac{1}{\eps_{\sublem{lem:C4}}^2m^2}\Big(\sum_{v\in V}\deg(v,U')^2
          -\sum_{v\in V}\deg(v,U') \Big)\,.
    \end{equation}
    For this observe that
    \begin{equation}\label{eq:regC4:b}
      b:=\frac1n\sum_{v\in V}\deg(v,U')
      =\frac1n d(U',V)\eps_{\sublem{lem:C4}} mn 
      = \Big(1\pm\frac{\eps_{\sublem{lem:C4}}^3}{10}\Big) 
             q\eps_{\sublem{lem:C4}} m \,.
    \end{equation}
    On the other hand, 
    $$
    b(V'):=\frac1{\eps_{\sublem{lem:C4}} n}\sum_{v\in V'}\deg(v,U')
          =\frac1{\eps_{\sublem{lem:C4}} n} d(U',V')\eps_{\sublem{lem:C4}}^2 mn
   $$
   and thus, by~\eqref{icecream}, we obtain that either
   \begin{align*}
     b(V')&>(1+\eps_{\sublem{lem:C4}})q\eps_{\sublem{lem:C4}} m\ge\Big(1+\frac{\eps_{\sublem{lem:C4}}}{2}\Big)b\\
      \intertext{or}
      b(V')&<(1-\eps_{\sublem{lem:C4}})q\eps_{\sublem{lem:C4}} m\le \Big(1-\frac{\eps_{\sublem{lem:C4}}}{2}\Big)b\,.
    \end{align*}
    Therefore Lemma~\ref{lem:CSdefect} applied with $k:=n$,
    $\delta:=\eps_{\sublem{lem:C4}}/2$, $\mu:=\eps_{\sublem{lem:C4}}$, and with~$b$
    instead of~$a$ implies that
    \begin{equation*}
      \sum_{v\in V}\deg(v,U')^2 
      \ge n \Big(1-\frac{\eps_{\sublem{lem:C4}}^3}{10}\Big)^2q^2 \eps_{\sublem{lem:C4}}^2 m^2\left(1+\frac{\eps_{\sublem{lem:C4}}^3}{4(1-\eps_{\sublem{lem:C4}})}\right)
      \ge \Big(1+\frac{\eps_{\sublem{lem:C4}}^3}{100}\Big)q^2 n \cdot \eps_{\sublem{lem:C4}}^2 m^2 \,.
    \end{equation*}
    Together with~\eqref{eq:regC4:aU'} and~\eqref{eq:regC4:b} this gives
    \begin{equation*}
      a(U')\ge \Big(1+\frac{\eps_{\sublem{lem:C4}}^3}{100}\Big)q^2 n -
      \Big(1+\frac{\eps_{\sublem{lem:C4}}^3}{10}\Big)\frac{q n}{\eps_{\sublem{lem:C4}} m}
      \ge \Big(1+\frac{\eps_{\sublem{lem:C4}}^3}{1000}\Big)q^2 n
    \end{equation*}
    as desired, where we used $q\ge400\eps_{\sublem{lem:C4}}^{-4}/m$.
  \end{claimproof}
\end{proof}

\section{Typical pairs, bad pairs, heavy pairs}
\label{sec:tbh}

The proofs of our inheritance lemmas rely on estimating the number of
copies of $C_4$ which use certain types of vertex pairs in one part of a
regular pair which is a subgraph of a bijumbled graph. We
will consider vertex pairs that are atypical for the regular pair, which we
call bad, and vertex pairs which are even atypical for the underlying
bijumbled graph, which we call heavy.

\begin{definition}[bad, heavy pairs]
  Let $G$ be a graph and~$U$ and~$V$ be disjoint vertex sets in~$G$. Let
  $q\in[0,1]$ and $\delta>0$.
  We say that a pair $uu'$ of distinct vertices in~$U$ is
  \emph{$(V,q,\delta)$-bad} in $G$ if 
  \[\deg_G(u,u';V)\ge(1+\delta)q^2|V|\,.\]
  Moreover, $uu'$ is \emph{$(V,q)$-heavy} in $G$ if 
  \[\deg_G(u,u';V)\ge 4q^2|V|\,.\]
  Pairs which are neither heavy nor bad (with certain parameters) will
  usually be called \emph{typical}.
\end{definition}

In a bijumbled graph we can establish good bounds on the number of copies 
of~$C_4$ which use heavy pairs.

\begin{lemma}[$C_4$-copies using heavy pairs]\label{lem:C4heavy}
  Let $\Gamma$ be a bipartite graph with partition classes $U$ and $V$ that is 
  $\big(p,c'p^{3/2}(\log_2 \tfrac{1}{p})^{-1/2}\sqrt{|U||V|}\big)$-bijumbled. 
  Assume further that for all $u\in U$ we have $\deg_{\Gamma}(u;V)\le 2p|V|$.  

 Then the number of copies of $C_4$ in~$\Gamma$ which use a pair in~$U$ which 
 is $(V,p)$-heavy in~$\Gamma$ is less than $64(c')^2p^4|U|^2|V|^2$.
\end{lemma}
\begin{proof} 
  We first fix $u\in U$ and count the number of copies of~$C_4$ in~$\Gamma$
  which use a~pair that contains~$u$ and is $(V,p)$-heavy.  Let
  $W_u\subset U\setminus\{u\}$ be the set of vertices $u'\in
  U\setminus\{u\}$ such that $uu'$ is a $(V,p)$-heavy pair.  We now
  split  $W_u$ according to the number of common neighbours the vertices 
  of~$W_u$ have with~$u$. Since $4p^2|V|\le\deg(u,u')\le 2p|V|$ for all 
  $u'\in W_u$, we can partition $W_u$ into 
  $W_u=S_1\dcup\dots\dcup S_{\lfloor\log_2 \tfrac{1}{p}\rfloor}$ with
  \[
   S_t=\big\{u'\in W_u \colon 2^{t-1}\cdot 4p^2|V|\le |N(u,u')|
       <2^t\cdot  4p^2|V|\big\}
  \]
  for $t=1,2,\ldots, \lfloor\log_2 \tfrac{1}{p}\rfloor$.  Since
  $\deg_{\Gamma}(u;V)\le 2p|V|$, we can take a superset $N_u\subset V$ of
  $N(u)$ of size $2p|V|$. Applying Lemma~\ref{lem:degbound} to $(U,N_u)$ 
  with $c'$ replaced by $c'(\log_2\frac1p)^{-1/2}$, $k=1$ and $\gamma=2^{t-1}$, 
  we see that the number of vertices in $U$ with at least
  $(1+\gamma)p|N_u|=(1+\gamma)\cdot 2p^2|V|$ neighbours in $N_u$ is at 
  most
  \[
   2(c')^2(\log_2\tfrac{1}{p})^{-1}2^{2-2t}|U|\,.
  \]
  Since each vertex of~$S_t$ has at least 
  $4\cdot 2^{t-1}p^2|V|\ge (1+\gamma)\cdot 2p^2|V|$ neighbours
  in $N(u)\subset N_u$, we conclude that
  \[
   |S_t| 
          \le 2^{3-2t}(c')^2|U|(\log_2 \tfrac{1}{p})^{-1}\,.
  \]

  For a fixed~$u'$, the number of copies of~$C_4$ using $u$ and $u'$ is 
  $\binom{|N(u,u')|}{2}\le\frac12|N(u,u')|^2$. Hence, the total number of 
  copies of $C_4$ using $u$ and any vertex of $S_t$ is at most
  \[
    |S_t|\frac12\big(2^t\cdot 4p^2|V|\big)^2
     \le
   64(c')^2 p^4|U| |V|^2(\log_2 \tfrac{1}{p})^{-1}\,.
  \]
  Summing over the at most $\log_2\tfrac{1}{p}$ values of $t$, we conclude that 
  the total number of copies of $C_4$ in $\Gamma$ using $u$ and some 
  $u'\in W_u$ is at most $64(c')^2p^4|U||V|^2$. 

  Finally, summing over all $u\in U$, the total number of copies of~$C_4$ in 
  $\Gamma$ using $(V,p)$-heavy pairs in $U$ is at most
  $64(c')^2p^4|U|^2|V|^2$ as desired.
\end{proof}

Using this lemma we obtain a good lower bound on the number of bad pairs in subgraphs of bijumbled graphs which are irregular or exceed a certain density.

\begin{lemma}[many bad pairs]
\label{lem:badpairs}
  Given $d\in(0,1)$ and $\eps^*\le 10^{-3}$, if $\delta\le(\eps^*)^{14}/10$,
  $\eps\le(\eps^*)^{14}d/100$ and $c'\le d^2\eps^{10}/100$ then for any
  $p\in(0,1/2)$ the following holds.\\
  Let $\Gamma$ be a graph and  let~$G$ be a bipartite subgraph of~$\Gamma$ 
  with vertex classes~$U$ and~$V$. Assume further that $\Gamma$ is 
  $\big(p,c'p^{3/2}(\log_2 \tfrac{1}{p})^{-1/2}\sqrt{|U||V|}\big)$-bijumbled and 
  $\deg_\Gamma(u;V)\le 2p|V|$ for all $u\in U$.
  If 
  \begin{enumerate}[label=\rom]
  \item\label{itm:badpairs:irreg} 
   $(U,V)$ has density at least $(d-\eps)p$ and is not $(\eps^*,p)$-regular in~$G$,
   or
  \item\label{itm:badpairs:dense}
   $(U,V)$ has density at least $(d+\eps^*)p$ in~$G$,
  \end{enumerate}
  then at least $(\eps^*)^{15} d^4|U|^2$ pairs $uu'\in\binom{U}{2}$ are
  $(V,dp,\delta)$-bad in~$G$.
\end{lemma}
\begin{proof}
  Let $P_h$ be the set of $(V, p)$-heavy pairs in~$\Gamma$ and $P_b$ be the
  set of $(V,dp,\delta)$-bad pairs in~$G$ which are not in~$P_h$. Let
  $P_t:=\binom{U}{2}\setminus (P_b\cup P_h)$. Denote by $C_4^h$ the number  
  of those copies of~$C_4$ in~$G$ that use a pair in $P_h$, and define $C_4^b$ 
  and $C_4^t$ similarly.

  We claim that, if~\ref{itm:badpairs:irreg} or~\ref{itm:badpairs:dense}
  are satisfied, then
  \begin{equation}\label{eq:badpairs:C4}
    C_4(G)\ge\big(1+(\eps^*)^{14}\big)d^4p^4\tfrac14|U|^2|V|^2\,.
  \end{equation}
  Indeed, 
  Lemma~\ref{lem:n0}  implies that~$U$ and~$V$ are of size at least
  $\frac18\big(c'(\log_2\tfrac{1}{p})^{-1/2}p\big)^{-2}\ge\frac18(c'p)^{-2}$\,,
  and hence
  $\nu_0:=\big(\min\{|U|,|V|\}\big)^{-1/2} \le \sqrt{8}c'p\,.$
  Now assume first that~\ref{itm:badpairs:irreg} holds. Then~$G$ has
  density at least 
  \begin{equation*}
    (d-\eps)p\ge (\eps^*)^{-10}\sqrt{8}c'p \ge (\eps^*)^{-10}\nu_0
  \end{equation*}
  because $c'\le d^2\eps^{10}/100$ and $\eps\le\eps^*$. But this is the
  condition we require to apply Lemma~\ref{lem:C4} with
  $\eps_{\sublem{lem:C4}}=\eps^*$ and $q=(d-\eps)p$ to~$G$. By
  Lemma~\ref{lem:C4}\ref{itm:C4:irreg} we have
  \begin{equation*}
    C_4(G)\ge\big(1+(\eps^*)^{13}\big)(d-\eps)^4p^4\tfrac14|U|^2|V|^2
    \ge(1+(\eps^*)^{14})d^4p^4\tfrac14|U|^2|V|^2\,,
  \end{equation*}
  where we used $\eps\le(\eps^*)^{14}d/100$ in the second
  inequality. Hence~\eqref{eq:badpairs:C4} holds in this case.
  If, on the other hand, \ref{itm:badpairs:dense} holds, then~$G$ has
  density at least $(d+\eps^*)p \ge dp \ge\eps^{-10}\sqrt{8}c'p \ge
  \eps^{-10}\nu_0$, where we used
  $c'\le d^2\eps^{10}/100$ in the second inequality. Hence
  Lemma~\ref{lem:C4}\ref{itm:C4:dense} applied with 
  $\eps_{\sublem{lem:C4}}=\eps$ and $q=(d+\eps^*)p$ gives
  \begin{equation*}
    C_4(G)\ge(1-\eps^8)(d+\eps^*)^4p^4\tfrac14|U|^2|V|^2
   \ge(1+(\eps^*)^{14})d^4p^4\tfrac14|U|^2|V|^2\,,
  \end{equation*}
  because $\eps\le(\eps^*)^{14}d/100$,
  which means we also get~\eqref{eq:badpairs:C4} in this case.

  Our next goal is to  obtain a lower bound for~$C_4^b$. 
  For this purpose observe that
  since $\deg_\Gamma(u;V)\le 2p|V|$ for all $u\in U$
  Lemma~\ref{lem:C4heavy} applies and we obtain
  $C_4^h \le 64(c')^2p^4|U|^2|V|^2$. Moreover,
  each pair $uu'\in P_t$ lies in at most $\binom{\deg_G(u,u';V)}{2}\le\binom{(1+\delta)d^2p^2|V|}{2}\le
  \frac{1}{2}(1+\delta)^2d^4p^4|V|^2$ copies of $C_4$ in $G$ by definition of~$P_t$. Hence
  \[C_4^t\le\binom{|U|}{2}\cdot\frac12(1+\delta)^2d^4p^4|V|^2\le\frac14(1+3\delta)d^4p^4|U|^2|V|^2\,.\]
  Thus we conclude from~\eqref{eq:badpairs:C4} that
  \begin{align*}
    C_4^b &=C_4(G)-C_4^h-C_4^t
    \\ & \ge\Big(\frac{(\eps^*)^{14}d^4}{4}-64(c')^2-\frac{3\delta
      d^4}{4}\Big)p^4|U|^2|V|^2 \\ &
    \ge 8(\eps^*)^{15}d^4p^4|U|^2|V|^2
    \,,
  \end{align*}
  where we use $\eps^*\le10^{-3}$, $c'\le d^2\eps^{10}/100$, 
  $\delta\le(\eps^*)^{14}/10$, and $\eps\leq\eps^*$.
  
  Now observe that each pair $uu'$ in $P_b$ is in at most
  $$\binom{\deg_\Gamma(u,u';V)}{2}\le\binom{4p^2|V|}{2}<8p^4|V|^2$$ copies of
  $C_4$ in $G$ by definition of~$P_b$. It follows that
  \begin{equation*}
    |P_b|\ge\frac{8(\eps^*)^{15}d^4p^4|U|^2|V|^2}{8p^4|V|^2}=(\eps^*)^{15}d^4|U|^2\,,
  \end{equation*}
  as desired.
\end{proof}

The next lemma provides an upper bound for the number of bad pairs in neighbourhoods.

\begin{lemma}[few bad pairs]\label{lem:fewbad} 
  Let $d,\delta>0$, let $c'\le \eps\le10^{-10}\delta^6d^8$, and
  $p\in(0,1)$. Let $G\subseteq\Gamma$ and let $U,V,W\subset V(\Gamma)$ be
  disjoint sets such that
  \begin{enumerate}[label=\rom]
  \item $(U,V)$ is $(p,c'p^{3/2}\sqrt{|U||V|})$-bijumbled in $\Gamma$,
  \item $(V,W)$ is $(p,c'p^2\sqrt{|V||W|})$-bijumbled in $\Gamma$, and
    $(\eps,d,p)$-regular in~$G$, and
  \item each $v\in V$ has $\deg_\Gamma(v;U)=(1\pm\eps)p|U|$.\label{fewbadiii}
  \end{enumerate}
  Then, for the sets $P_b(u)$ of pairs $vv'\in\binom{N_\Gamma(u;V)}{2}$ which
  are $(W,dp,\delta)$-bad in~$G$, we have
  $\sum_{u\in U}\big|P_b(u)\big|\le \delta p^2 |U||V|^2$.
\end{lemma}
\begin{proof}
Let $P_b$ be the set of all pairs $vv'\in\binom{V}{2}$ which are
$(W,dp,\delta)$-bad in~$G$. Our first step is to obtain an upper bound on $|P_b|$.
 \begin{claim}\label{clm:fewbad}
   $|P_b|\le \tfrac12\delta\binom{|V|}{2}$.
  \end{claim}
  \begin{claimproof}[Proof of Claim~\ref{clm:fewbad}]
    We conclude from Lemma~\ref{lem:n0} applied to $(V,W)$
    that
    $|V|,|W|\ge\tfrac18(c')^{-2}p^{-3}\ge
    \tfrac{1}{8}p^{-2}\eps^{-2}$. 
    This implies
    \begin{equation}\label{eq:fewbad:sizeVW}
      |V|-1\ge(1-\eps)|V| \quad\text{and}\quad
      (d-\eps)p|V|-1\ge(1-\eps)(d-\eps)p|V|\,,
    \end{equation}
    which we will use to estimate binomial coefficients.

    Let~$\mu$ be such that $|P_b|=\mu\binom{|V|}{2}$. Our goal is to get an
    upper bound on~$\mu$. For this purpose we shall first use the defect form of
    Cauchy-Schwarz, Lemma~\ref{lem:CSdefect}, to get a lower bound on the
    number of $C_4$-copies in $(V,W)$ in terms of~$\mu$. Then we combine this
    bound with the upper bound on the number of $C_4$-copies in regular pairs
    provided by Lemma~\ref{lem:c4count}.

    For the application of Lemma~\ref{lem:CSdefect}
    set $a_{vv'}:=\deg_G(v,v';W)$ for each $vv'\in\binom{V}{2}$, and
    define
    \begin{equation*}
      a':=\binom{|V|}{2}^{-1}\sum_{vv'\in\binom{V}{2}}a_{vv'}=\binom{|V|}{2}^{-1}\sum_{w\in
      W}\binom{\deg_G(w;V)}{2}
    \end{equation*}
    to be the average of the $a_{vv'}$. Let us now first establish some
    bounds on~$a'$.
    Since $(V,W)$ is $(\eps,d,p)$-regular in $G$, all but at most $\eps|W|$
    vertices of $W$ have at least $(d-\eps)p|V|$ neighbours in $V$. This gives
    \begin{equation}\label{eq:fewbad:lowera}
      a'\ge\binom{|V|}{2}^{-1}(1-\eps)|W|\binom{(d-\eps)p|V|}{2}
      \geByRef{eq:fewbad:sizeVW} (1-\eps)^2(d-\eps)^2p^2|W|=:a\,.
    \end{equation}
    On the other hand, by Lemma~\ref{lem:degbound}, the number of vertices
    $w\in W$ with $\deg_\Gamma(w;V)>2p|V|$ is at most $2(c'p)^2|W|$. We
    conclude that
    \begin{equation}\label{eq:fewbad:uppera}
     \begin{aligned}
    a'&\le\binom{|V|}{2}^{-1}\bigg(|W|\binom{2p|V|}{2}+2(c'p)^2|W|\binom{|V|}{2}\bigg)\\
   &\le \big(4+2(c')^2\big)p^2|W|\le 5p^2|W|\,.
    \end{aligned}
    \end{equation}
    Now we apply Lemma~\ref{lem:CSdefect} with
    $k=\binom{|V|}{2}$ and $a$, $\delta$ and $\mu$ as given.
    By~\eqref{eq:fewbad:lowera} the $a_{vv'}$ average at least~$a$.
    Moreover, by definition all $\mu k$ pairs $vv'\in P_b$ are $(W,dp,\delta)$-bad
    in~$G$, that is, $a_{vv'}\ge(1+\delta)d^2p^2|W|\ge (1+\delta)a$ by~\eqref{eq:fewbad:lowera}.
    Lemma~\ref{lem:CSdefect} thus guarantees that
   \begin{align*}
      \sum_{vv'\in\binom{V}{2}}a_{vv'}^2 &\ge
      ka^2\Big(1+\frac{\mu\delta^2}{1-\mu}\Big)\\ &\ge\binom{|V|}{2}(1-\eps)^4(d-\eps)^4p^4|W|^2(1+\mu\delta^2)\\
     & \geByRef{eq:fewbad:sizeVW} \tfrac12(1-\eps)^5(d-\eps)^4(1+\mu\delta^2)p^4|V|^2|W|^2\\ &
      \ge\tfrac12(d-3\eps)^4(1+\mu\delta^2)p^4|V|^2|W|^2\,,
    \end{align*}
    since $(1-\eps)^5(d-\eps)^4\geq (d-3\eps)^4$.

    Hence the number of copies of $C_4$ in $G[V,W]$ is 
    \begin{align*}
      \sum_{vv'\in\binom{V}{2}}\binom{a_{vv'}}{2} &=\tfrac12\sum_{vv'\in\binom{V}{2}}a_{vv'}^2-\tfrac12\sum_{vv'\in\binom{V}{2}}a_{vv'}\\ &
      \geByRef{eq:fewbad:uppera} \tfrac14(d-3\eps)^4(1+\mu\delta^2)p^4|V|^2|W|^2-|V|^2\tfrac54p^2|W|\\ &
      \ge \tfrac14(d-4\eps)^4(1+\mu\delta^2)p^4|V|^2|W|^2\,,
    \end{align*}
    where we used $|W|\ge \tfrac{1}{8}p^{-2}\eps^{-2}$ in the last inequality.
    On the other hand, since $(V,W)$ is $(\eps,d,p)$-regular in $G$, and
    $(p,c'p^2\sqrt{|V||W|})$-bijumbled in $\Gamma$,
    Lemma~\ref{lem:c4count} implies that the of copies of $C_4$ in $G[V,W]$ is
    at most $\tfrac14\big(d^4+100(c'+\eps)^{1/2}\big)p^4|V|^2|W|^2$.
    Putting these two inequalities together we obtain
    \[(d-4\eps)^4(1+\mu\delta^2)\le d^4+100(c'+\eps)^{1/2}\,.\]
    Using the assumption that $c'\le \eps\le10^{-10}\delta^6d^8$, we deduce that $\mu\le \tfrac12\delta$ as desired.
  \end{claimproof}
  
  Now for each $v\in V$, let $V_v:=\{v'\in V\colon vv'\in P_b\}$. Note that
  $vv'\in P_b(u)$ if and only if $vv'\in P_b$ and $u\in N_\Gamma(v;U)$ and $uv'\in
  E(\Gamma)$. It follows that
  \[\sum_{u\in U} |P_b(u)|=\tfrac12\sum_{v\in V}e_\Gamma\big(V_v,N_\Gamma(v;U)\big)\,.\]
  Since $(U,V)$ is $(p,c'p^{3/2}\sqrt{|U||V|})$-bijumbled in $\Gamma$, we
  have for each $v\in V$ that
  \begin{align*}
    e_\Gamma\big(V_v,N_\Gamma(v;U)\big)&\le p|V_v|\deg_\Gamma(v;U)+c'p^{3/2}\sqrt{|U||V|}\sqrt{|V_v|\deg_\Gamma(v;U)}\\
    &\le (1+\eps)p^2|V_v||U|+c'p^{3/2}\sqrt{|U||V|}\sqrt{(1+\eps)p|U||V|}\\
    &= \big((1+\eps)|V_v|+c'\sqrt{1+\eps}|V|\big)p^2|U|\,,
  \end{align*}
  where we  use assumption~\ref{fewbadiii} for the second inequality.
  We therefore obtain
  \begin{align*}
    \sum_{u\in U} |P_b(u)|&\le\tfrac12\sum_{v\in V}\big((1+\eps)|V_v|+c'\sqrt{1+\eps}|V|\big)p^2|U|\\
    &\le (1+\eps)p^2|P_b||U|+c'p^2|V|^2|U|\\
    &\le \tfrac{1}{2}(1+\eps)\delta p^2\binom{|V|}{2}|U|+c'p^2|V|^2|U|
    \\ &\le \delta p^2|U||V|^2,
  \end{align*}
  as desired, where in the third inequality we use Claim~\ref{clm:fewbad}.
\end{proof}

\section{One-sided inheritance}
\label{sec:oneside}

To prove Lemma~\ref{lem:oneside} we combine Lemma~\ref{lem:badpairs} and
Lemma~\ref{lem:fewbad}. The former asserts that any
vertex~$x$ such that $\big(N_\Gamma(x;Y),Z\big)$ is not
$(\eps',d,p)$-regular in~$G$ creates many pairs in $N_\Gamma(x,Y)$ which
are bad in $(Y,Z)\subset G$, whereas the latter upper bounds the sum over
$x\in X$ of the number of such bad pairs.

\begin{proof}[Proof of Lemma~\ref{lem:oneside}]
  We may assume without loss of generality that $0<\eps'<10^{-4}$. Given in addition $d>0$ set
  \begin{equation}\label{eq:oneside:epsc}
    \delta=10^{-10}(\eps')^{20}d^4\,,\quad
    \eps=10^{-16}(\eps')^{22}d^{16}\delta^6\quad\text{and}\quad
    c=10^{-4}\eps^{10}d^4\delta^2\,.
  \end{equation}
  
  As a preparation we first `clean up' the partition classes $X$, $Y$, $Z$ as follows. We let $Y'\subset Y$ be
  the set of vertices $y$ of $Y$ with
  \begin{equation}\label{eq:oneside:propY'}
    \begin{aligned}
    \deg_\Gamma(y;X)&=(1\pm\eps)p|X|\,, &
    \deg_\Gamma(y;Z)&=(1\pm\eps)p|Z|\,, \quad\text{and} \\
    &&\deg_G(y;Z)&=(d\pm\eps)p|Z|\,.
    \end{aligned}
  \end{equation}
  Observe that by Lemma~\ref{lem:degbound} and by $(\eps,d,p)$-regularity
  of $(X,Y)$ in $G$ we have
  \begin{equation}\label{eq:oneside:sizeY}
    |Y\setminus Y'| \le
    2c^2p\eps^{-2}|Y|+2c^2p^2\eps^{-2}|Y|+2\eps|Y|\leByRef{eq:oneside:epsc}3\eps|Y|\,.
  \end{equation}
  Hence, $(X,Y')$ is $(p,\frac32cp^{3/2}\sqrt{|X||Y'|})$-bijumbled
  in~$\Gamma$.
  We then let $X'\subset X$ be the set of vertices $x$ of $X$ with 
  \begin{equation}\label{eq:oneside:degxY'}
    \deg_\Gamma(x;Y')=(1\pm\eps)p|Y'|\quad\text{and}\quad\deg_\Gamma(x;Y\setminus Y')\le 4\eps p|Y|\,. 
  \end{equation}
  Similarly as before, we apply Lemma~\ref{lem:degbound} once to $(X,Y')$
  with $\gamma=\eps$ and once to the pair $(X,Y\setminus Y')$ in $(X,Y)$ with $\gamma=\frac 13$
  and use~\eqref{eq:oneside:epsc} and ~\eqref{eq:oneside:degxY'} to obtain
  \begin{equation}\label{eq:oneside:sizeX'}
    |X\setminus X'|\le 2(\tfrac32c)^2p\eps^{-2}|X|+2\big(3c\big)^2p|X|\le\eps p|X|\,.
  \end{equation}
  By Lemma~\ref{lem:slicing} and because of~\eqref{eq:oneside:propY'} and~\eqref{eq:oneside:sizeX'} it follows that 
  \begin{equation}\label{eq:oneside:moreY}
    \text{$(Y',Z)$ is $(2\eps,d,p)$-regular in $G$}
    \qquad\text{and}\qquad
    \deg_\Gamma(y;X')=(1\pm 3\eps)p|X'|
  \end{equation}
  for each $y\in Y'$.
  Moreover, $(Y',Z)$ is
  $\big(p,\frac32cp^2(\log_2\tfrac{1}{p})^{-1/2}\sqrt{|Y'||Z|}\big)$-bijumbled
  in $\Gamma$. Thus for each $x\in X'$, because $|Y'|\le\deg_\Gamma(x;Y')/(p(1-\eps))$ by
  \eqref{eq:oneside:degxY'}, 
  \begin{equation}\label{eq:oneside:NxYZ}
    (Y',Z)\text{ is }
    \big(p,2cp^{3/2}(\log_2\tfrac{1}{p})^{-1/2}\sqrt{|N_\Gamma(x;Y')||Z|}\big)\text{-bijumbled in $\Gamma$}\,.
  \end{equation}
  Finally, let $X^*$ be the set of vertices in $X'$ such that
  $\big(N_\Gamma(x;Y'),Z\big)$ is not
  $\big(\tfrac{\eps'}{2},d,p\big)$-regular in $G$.

  We claim that $\big(N_\Gamma(x;Y),Z\big)$ is $(\eps',d,p)$-regular in~$G$
  for all $x\in X'\setminus X^*$.  In order to show this we apply 
  Lemma~\ref{lem:sticking} with $\eps_{\sublem{lem:sticking}}=\frac12\eps'$
  and $c_{\sublem{lem:sticking}}=2c$, and with $U=N_\Gamma(x;Y')$,
  $U'=N_\Gamma(x,Y)$ and $V=V'=Z$. This is possible
  by~\eqref{eq:oneside:NxYZ},  the definition of~$X^*$, and
  because
  \begin{equation}\label{eq:onesided:U'U}
    \begin{split}
    |U'\setminus U|& \le\deg_\Gamma(x;Y\setminus Y') 
    \leByRef{eq:oneside:degxY'} 4\eps p|Y| 
    \leByRef{eq:oneside:epsc} \frac 1{10}\Big(\frac{\eps'}2\Big)^3 (1-3\eps)p|Y| \\ &
    \leq \frac 1{10}\Big(\frac{\eps'}2\Big)^3 \deg_\Gamma(x;Y) 
    \leq \frac 1{10}\Big(\frac{\eps'}2\Big)^3 |U|\,,
    \end{split}
  \end{equation}
  where for the second to last inequality we use \eqref{eq:oneside:sizeY} and \eqref{eq:oneside:degxY'}.
  We conclude that indeed $\big(N_\Gamma(x;Y),Z\big)$ is
  $(\eps',d,p)$-regular in~$G$. Therefore, by~\eqref{eq:oneside:sizeX'} it
  suffices to show that $|X^*|\le\tfrac12\eps'|X|$ to complete the proof.

  For this purpose, we define for each $x\in X'$
  \[P_b(x):=\Big\{yy'\in\binom{N_\Gamma(x;Y')}{2}\colon yy'\text{ is
    $(Z,dp,\delta)$-bad in }G\Big\}\,.\] and determine a lower bound on
  $\sum_{x\in X'}|P_b(x)|$ in terms of~$|X^*|$ with the help of Lemma~\ref{lem:badpairs}
  and an upper bound in terms of~$|X'|$ with the help of
  Lemma~\ref{lem:fewbad}.

  For the lower bound, fix $x\in X^*$.  By~\eqref{eq:oneside:propY'} the
  density of $\big(N_\Gamma(x;Y'),Z\big)$ in~$G$ is at least $(d-\eps)p$.
  Hence, by~\eqref{eq:oneside:epsc}, \eqref{eq:oneside:propY'},
  \eqref{eq:oneside:NxYZ} and the definition of~$X^*$ we may apply
  Lemma~\ref{lem:badpairs} with parameters $d$, $\eps^*=\tfrac{\eps'}{2}$,
  $\delta$, $\eps_{\sublem{lem:badpairs}}=2\eps$, $c'=2c$ and~$p$ to the
  pair $\big(N_\Gamma(x;Y'),Z\big)$ in $G$, in the bijumbled graph $(Y,Z)$ in $\Gamma$, using
  condition~\ref{itm:badpairs:irreg} of this lemma. We obtain $|P_b(x)|\ge
  \big(\tfrac{\eps'}{2}\big)^{15}d^4\deg_\Gamma(x;Y')^2$, and therefore
  \begin{equation}\label{eq:oneside:lowerPb}
   \sum_{x\in X'}|P_b(x)|\ge\sum_{x\in X^*}|P_b(x)|
   \geByRef{eq:oneside:degxY'}|X^*|\cdot \Big(\frac{\eps'}{2}\Big)^{15}d^4(1-\eps)^2p^2|Y'|^2\,.
  \end{equation}
  For the upper bound we use Lemma~\ref{lem:fewbad} with input $\delta$,
  $\eps_{\sublem{lem:fewbad}}=3\eps$, and $c'=2c$, and setting $U=X'$,
  $V=Y'$ and $W=Z$, which we may do
  by~\eqref{eq:oneside:epsc},~\eqref{eq:oneside:sizeY},~\eqref{eq:oneside:sizeX'}
  and~\eqref{eq:oneside:moreY}. The conclusion is that $\sum_{x\in
    X'}|P_b(x)|\le\delta p^2|X'||Y'|^2$. Together
  with~\eqref{eq:oneside:lowerPb} this gives
  $\big(\tfrac{\eps'}{2}\big)^{15}d^4(1-\eps)^2|X^*|\le \delta
  |X'|\le\delta|X|$ and therefore by~\eqref{eq:oneside:epsc} we indeed have
  $|X^*|\le\tfrac12\eps'|X|$.
\end{proof}

\section{Two-sided inheritance}
\label{sec:twoside}

The proof of Lemma~\ref{lem:twoside} follows a similar pattern to that of Lemma~\ref{lem:oneside}.

\begin{proof}[Proof of Lemma~\ref{lem:twoside}]
  Assume without loss of generality that $0<\eps'<10^{-4}$. Given $d>0$, we set
  \begin{equation}\label{eq:twoside:epsc}
    \begin{aligned}
      \eps^*&=10^{-20}(\eps')^{14}d\,,\quad &\delta&=10^{-20}d^4(\eps^*)^{31}\,,\\
      \eps&=10^{-20}(\eps^*)^{30}\delta^6d^8\qquad\text{and}\quad &c&=10^{-3}d^2\eps^{10}\delta\,.
    \end{aligned}
  \end{equation}
  We now `clean up' the partition classes $X$, $Y$, $Z$ as follows. First, let $Y'\subseteq Y$ be the set of vertices $y\in Y$ with
  \begin{equation}\label{eq:twoside:propY'}
    \deg_\Gamma(y;Z)=(1\pm\eps)p|Z|\,,\quad\text{and}\quad \deg_\Gamma(y;X)=(1\pm\eps)p|X|\,.
  \end{equation}
  By Lemma~\ref{lem:degbound} and~\eqref{eq:twoside:epsc} we have
  \begin{equation}\label{eq:twoside:sizeY'}
    |Y\setminus Y'|\le 2c^2(\log_2\tfrac{1}{p})^{-1}p^3\eps^{-2}|Y|+2c^2p^2\eps^{-2}|Y|\le\eps|Y|\,.
  \end{equation}
  We let $X'\subset X$ be the set of vertices $x\in X$ with
  \begin{equation}\label{eq:twoside:propX'}\begin{split}
    &\deg_\Gamma(x;Y')=(1\pm\eps)p|Y'|\,, \quad
    \deg_\Gamma(x;Y\setminus Y')\le 2\eps p|Y|\,, \quad\text{and} \\
    &\deg_\Gamma(x;Z)=(1\pm\eps)p|Z|\,.
  \end{split}
  \end{equation}
  Again, by Lemma~\ref{lem:degbound} and \eqref{eq:twoside:epsc} we have
  \begin{equation}\label{eq:twoside:sizeX'}
    |X\setminus X'|
     \le 
   2\cdot8c^2p^2\eps^{-2}|X| + 2c^2p^4\eps^{-2}|X|
     \le
   \eps p|X|\,.
  \end{equation}
  By Lemma~\ref{lem:slicing}, by~\eqref{eq:twoside:propY'} and by~\eqref{eq:twoside:sizeX'}, we obtain
  \begin{equation}\label{eq:twoside:moreY}
    (Y',Z)\text{ is $(2\eps,d,p)$-regular in }G
    \quad \text{and }\quad 
    \deg_\Gamma(y;X')=(1\pm3\eps)p|X'|
  \end{equation}
  for each $y\in Y'$. 
  Moreover, $(Y', Z)$ is
  $\big(p,2cp^{5/2}\big(\log_2\tfrac1p)^{-1/2}\sqrt{|Y'||Z|}\big)$-bijumbled
  in $\Gamma$. 
  Thus for each $x\in X'$, because $|Y'|\le\deg_\Gamma(x;Y')/(p(1-\eps))$ 
  and $|Z|\le\deg_\Gamma(x;Z)/(p(1-\eps))$ by
  \eqref{eq:twoside:propX'}, 
  \begin{equation}\label{eq:twoside:NxYZ}
    (Y', Z)\text{ is }
    \big(p,2cp^{3/2}(\log_2\tfrac{1}{p})^{-1/2}
    \sqrt{|N_\Gamma(x;Y')||N_\Gamma(x;Z)|}\big)\text{-bijumbled in $\Gamma$}\,.
  \end{equation}
  For $x\in X'$ let
  \begin{equation*}
    Y_x:=N_\Gamma(x;Y') \quad\text{and}\quad Z_x:=N_\Gamma(x;Z)\,.
  \end{equation*}
  Define
  \begin{align*}
    X^*_ 1&:=\big\{x\in X'\colon d_G(Y_x,Z_x)\ge (d-\eps^*)p \text{ and
    } (Y_x,Z_x)_G \text{ is not $\big(\tfrac{\eps'}{2},d,p\big)$-regular}
    \big\}\,, \\
    X^*_ 2&:=\big\{x\in X'\colon d_G(Y_x,Z_x)\ge \big(d+(\eps^*)^2\big)p \big\}\,,
  \end{align*}
  and let $X^*:=X^*_1\cup X^*_2$.  Finally, let $X^{**}$ be the set of $x\in
  X'\setminus X^*$ such that $\big(Y_x,Z_x\big)$ has
  density less than $(d-\eps^*)p$ in~$G$.

  We claim that $\big(N_\Gamma(x;Y),Z_x\big)$ is
  $(\eps',d,p)$-regular in~$G$ for all $x\in X'\setminus(X^*\cup X^{**})$.
  This again follows from
  Lemma~\ref{lem:sticking}, which we apply with $\eps_{\sublem{lem:sticking}}=\frac12\eps'$
  and $c_{\sublem{lem:sticking}}=2c$, and with $U=Y_x$,
  $U'=N_\Gamma(x,Y)$, $V=V'=Z_x$. This is possible  by~\eqref{eq:twoside:NxYZ}, 
  because $\big(Y_x,Z_x\big)$ is
  $(\frac12\eps',d,p)$-regular in~$G$ by the definition of~$X^*$ and~$X^{**}$,
  and
  because $|U'|\le \big(1+\tfrac{1}{10}(\tfrac{1}{2}\eps')^3\big)|U|$ by a
  calculation analogous to~\eqref{eq:onesided:U'U}.
  We conclude that indeed $\big(N_\Gamma(x;Y),Z_x\big)$ is
  $(\eps',d,p)$-regular in~$G$. Therefore, by~\eqref{eq:twoside:sizeX'} it
  suffices to show that $|X^*|\le\tfrac13\eps'|X|$ and $|X^{**}|\le\tfrac13\eps'|X|$ to complete the proof.
  
  We start with the former. For each $x\in X'$, let
  \[P_b^*(x):=\Big\{yy'\in\binom{Y_x}{2}\colon
     yy'\text{ is $\big(Z_x,dp,\delta\big)$-bad in
    $G$}\Big\}\,.\] 
  To bound $|X^*|$, we will again estimate $\sum_{x\in
    X'}|P_b^*(x)|$ in two different ways. The first part is given by the
  following claim.
 
  \begin{claim}\label{clm:twoside:lower}
     $\sum_{x\in X'}|P_b^*(x)|\ge\sum_{x\in X^*}|P_b^*(x)|\ge 10^{-10}(\eps^*)^{30}d^4p^2|X^*||Y'|^2$.
   \end{claim}
  \begin{claimproof}
    This bound will follow from Lemma~\ref{lem:badpairs}.  We first need to
    `clean up' the pairs $(Y_x,Z_x)$ for the application of this lemma.
    Let $Y'_x\subseteq Y_x$ consist of the vertices
    $y\in Y_x$ with $\deg_\Gamma(y;Z_x)\le 2p|Z_x|$.
    The pair $(Y_x,Z_x)$ is
    $(p,2cp^{3/2}\big(\log_2\tfrac1p\big)^{-1/2}\sqrt{|Y_x||Z_x|})$-bijumbled
    since $|Y_x||Z_x|=(1\pm\eps)^2p^2|Y||Z|$ by~\eqref{eq:twoside:NxYZ}.
    So Lemma~\ref{lem:degbound} and~\eqref{eq:twoside:epsc} imply
    $|Y_x\setminus Y'_x|\le 8c^2\big(\log_2\tfrac1p\big)^{-1}p|Y_x|\le\eps p|Y_x|$.
    Moreover,
    \begin{equation}\label{eq:twoside:YxZx}
      (Y'_x,Z_x)\text{ is }\big(p,4cp^{3/2}\big(\log_2\tfrac1p\big)^{-1/2}\sqrt{|Y'_x||Z_x|}\big)\text{-bijumbled}\,.
    \end{equation}

    We now first consider vertices $x\in X^*_1$. To bound $|P_b^*(x)|$ we want
    to apply Lemma~\ref{lem:badpairs} to $(Y'_x,Z_x)$, using
    condition~\ref{itm:badpairs:irreg} of Lemma~\ref{lem:badpairs}. For
    this purpose we will first show that $(Y'_x,Z_x)$ is also not
    $\big(\tfrac{\eps'}{4},d,p\big)$-regular in $G$. Indeed, 
    by~\eqref{eq:twoside:epsc} and~\eqref{eq:twoside:YxZx} we can apply the
    contrapositive of Lemma~\ref{lem:sticking} with
    $\eps_{\sublem{lem:sticking}}=\tfrac{\eps'}{4}$ and
    $c_{\sublem{lem:sticking}}=4c$, and with $U=Y'_x$, $U'=Y_x$, $V=V'=Z_x$
    because $$|Y_x\setminus Y'_x|\le\eps p|Y_x|\le2\eps
    p|Y'_x|\le\frac1{10}(\eps'/4)^3|Y'_x|.$$  Since $(Y_x,Z_x)$ is not
    $\big(\tfrac{\eps'}{2},d,p\big)$-regular in $G$, this lemma implies
    that $(Y'_x,Z_x)$ is also not $\big(\tfrac{\eps'}{4},d,p\big)$-regular
    in $G$ as claimed.
    Hence, by ~\eqref{eq:twoside:epsc}, \eqref{eq:twoside:YxZx}, the definition of $X^*_1$ and the
    definition of~$Y'_x$ we may apply Lemma~\ref{lem:badpairs} to
    $(Y'_x,Z_x)$ with input $d$,
    $\eps^*_{\sublem{lem:badpairs}}=\tfrac{\eps'}{4}$, $\delta$,
    $\eps_{\sublem{lem:badpairs}}=\eps^*$, $c'=4c$ and~$p$. We conclude that
    \begin{equation*}
      \big|P_b^*(x)\big|\ge\big(\tfrac{\eps'}{4}\big)^{15}d^4|Y'_x|^2
      \geBy{\eqref{eq:twoside:epsc},\eqref{eq:twoside:propX'}}
      10^{-10}(\eps^*)^{30}d^4p^2|Y'|^2\,.
    \end{equation*}

    It remains to consider $x\in X^*_2$. In this case we want to use
    Lemma~\ref{lem:badpairs}\ref{itm:badpairs:dense}. To obtain the
    required density condition, observe that
    \begin{align*}
      e_\Gamma(Y_x\setminus Y'_x,Z_x)
     & \le p \cdot \eps p|Y_x||Z_x|+2cp^{3/2}\big(\log_2\tfrac1p\big)^{-1/2}\sqrt{|Y_x||Z_x|}\sqrt{\eps
        p|Y_x||Z_x|} \\
     & \leByRef{eq:twoside:epsc}\tfrac13(\eps^*)^2 p^2|Y_x||Z_x|\,.
    \end{align*}
    Since $d_G(Y_x,Z_x)\ge\big(d+(\eps^*)^2\big)p$, it follows that 
    $d_G(Y'_x,Z_x)\ge\big(d+\tfrac13(\eps^*)^2\big)p$ by~\eqref{eq:twoside:epsc}.
    Hence,  by~\eqref{eq:twoside:epsc} and~\eqref{eq:twoside:YxZx} we can apply Lemma~\ref{lem:badpairs} to
    $(Y'_x,Z_x)$ with input $d$,
    $\eps^*_{\sublem{lem:badpairs}}=\tfrac13(\eps^*)^2$, $\delta$,
    $\eps_{\sublem{lem:badpairs}}=\eps$ and $c'=4c$ and conclude that
    \begin{equation*}
      P_b^*(x)
      \ge\big(\tfrac{(\eps^*)^2}{3})^{15}d^4|Y'_x|^2
      \geByRef{eq:twoside:epsc}
      10^{-10}(\eps^*)^{30}d^4p^2|Y'|^2.
    \end{equation*}
    Summing over all $x\in X^*=X^*_1\cup X^*_2$, the claim follows.
  \end{claimproof}
  
  The next claim establishes a complementing upper bound for $\sum_{x\in X'}|P_b^*(x)|$.
 
  \begin{claim}\label{clm:twoside:upper}
    $\sum_{x\in X'}|P_b^*(x)|\le\delta p^2|X||Y'|^2$.
  \end{claim}
  \begin{claimproof}
    In order to estimate $\sum_{x\in X'}|P_b^*(x)|$ we will distinguish
    between the contribution made to this sum by the pairs
    \[P_b:=\Big\{yy'\in\binom{Y'}{2}\colon yy'\text{ is
      $\big(Z,dp,\tfrac{\delta}{2}\big)$-bad in $G$}\Big\}\,\] and that
    made by the pairs $P_t:=\binom{Y'}{2}\setminus P_b$.
    
    For the former let $P_b(x):=\{yy'\in\binom{Y_x}{2}\colon yy' \text{ is
    }(Z,dp,\frac\delta2)\text{-bad in }G\}$. We use the very rough bound 
    \begin{equation*}
      \sum_{x\in X'}|P_b^*(x)\cap P_b|
      \le\sum_{x\in X'} |P_b(x)|
      \,,
    \end{equation*}
    which holds since $P^*_b(x)\cap P_b\subseteq P_b(x)$ for all $x\in X'$.
    By Lemma~\ref{lem:fewbad} applied to $X',Y',Z$ with parameters~$d$,
    $\delta_{\sublem{lem:fewbad}}=\tfrac12\delta$, $c'=2c$,
    $\eps_{\sublem{lem:fewbad}}=3\eps$, which we can do 
    by~\eqref{eq:twoside:epsc},
    \eqref{eq:twoside:moreY} and
    since $(Y',Z)$ is $(2\eps,d,p)$-regular in $G$, we thus have
	  \begin{equation}\label{eq:twoside:upperPbb}
	    \sum_{x\in X'}|P_b^*(x)\cap P_b|
        \le \sum_{x\in X'} |P_b(x)|
        \le\tfrac12\delta p^2|X'||Y'|^2\le\tfrac12\delta p^2|X||Y'|^2\,.
	  \end{equation}

	For the contribution of $P_t$ on the other hand,  define 
      \begin{equation*}
        V_b(yy'):=\{x\in X'\colon
        \deg_\Gamma\big(x,N_G(y,y';Z)\big)\ge(1+\delta)d^2p^2|Z_x|\}
      \end{equation*}
      for $yy'\in\binom{Y'}{2}$. Observe  that we have $yy'\in
      P_b^*(x)$ for some $x\in X'$ if and only if $x\in N_\Gamma(y,y';X')$
      and $x\in V_b(yy')$.  It follows that
      \begin{equation*}
        \sum_{x\in X'}|P_b^*(x)\cap P_t| \le \sum_{yy'\in P_t} |V_b(yy')|\,.
      \end{equation*}
      Now let $yy'\in P_t$ be fixed. We have
      $\deg_G(y,y';Z)\le\big(1+\tfrac12\delta\big)d^2p^2|Z|$ by definition
      of $P_t$. Let $Z_{yy'}$ be a superset of $N_G(y,y';Z)$ of size
      $\big(1+\tfrac12\delta\big)d^2p^2|Z|$. By assumption, $(X,Z)$ is
      $(p,cp^3\sqrt{|X||Z|})$-bijumbled, and so $(X,Z_{yy'})$ is
      $(p,cd^{-1}p^2\sqrt{|X||Z_{yy'}|})$-bijumbled.
      Lemma~\ref{lem:degbound}, with parameters $\gamma=\eps$,
      $c'=cd^{-1}$, $k=2$,  then gives
      \begin{equation}\label{eq:twoside:degbound}
        \big|\{x\in X'\colon \deg_\Gamma(x;Z_{yy'})\ge(1+\eps)p|Z_{yy'}|\}\big|
        \le 2c^2d^{-2}p^2\eps^{-2}|X|\,.
      \end{equation}
      Since
      \begin{equation*}
        (1+\delta)d^2p^2|Z_x|
        \geByRef{eq:twoside:propX'}(1+\delta)d^2p^2(1-\eps)p|Z|
        \geByRef{eq:twoside:epsc}(1+\eps)p|Z_{yy'}|\,
      \end{equation*}
      and by the choice of $Z_{yy'}$, the left-hand side of~\eqref{eq:twoside:degbound} is at least
      $|V_b(yy')|$.
      Summing over all $yy'\in P_t$ we conclude
	  \begin{equation*}
	    \sum_{x\in X'}|P_b^*(x)\cap P_t|
        \le \sum_{yy'\in P_t} |V_b(yy')|
        \le 2c^2d^{-2}p^2\eps^{-2}|X||Y'|^2
        \leByRef{eq:twoside:epsc}\tfrac12\delta p^2|X||Y'|^2\,.
	  \end{equation*}
	  Together with~\eqref{eq:twoside:upperPbb} this proves the claim.
  \end{claimproof}

  Claims~\ref{clm:twoside:lower} and~\ref{clm:twoside:upper} imply
  $10^{-10}(\eps^*)^{30}d^4p^2|X^*||Y'|^2\le\delta p^2|X||Y'|^2$
  and hence
  \begin{equation}\label{eq:twoside:X*}
    |X^*|\le 10^{10}\delta d^{-4}(\eps^*)^{-30}|X|
    \le\eps^*|X|
    \le\tfrac13\eps'|X|\,,
  \end{equation}
  by~\eqref{eq:twoside:epsc}.
  It remains to bound $|X^{**}|$. Let $X'':=X'\setminus X^*$.

 \begin{claim}\label{clm:twoside:rightdensity}
    $|X^{**}|\le 2\eps^*|X''|\le\tfrac13\eps'|X|$.
  \end{claim}
  \begin{claimproof}
    Let $\mu:=|X^{**}||X''|^{-1}$.
    We bound~$\mu$ by considering the number~$T$ of triples $xyz$ with
    $x\in X''$, $y\in Y'$, $z\in Z$ which are such that $xy,xz\in
    E(\Gamma)$ and $yz\in E(G)$. Observe that $T=\sum_{x\in X''}e_G\big(Y_x,Z_x\big)$
    and
    \begin{multline*}
      e_G(Y_x,Z_x)
      =d_G(Y_x,Z_x)\deg_\Gamma(x;Y')\deg_\Gamma(x;Z)
      \leByRef{eq:twoside:propX'}d_G\big(Y_x,Z_x\big)(1+\eps)^2p^2|Y'||Z|
    \end{multline*}
    for each $x\in X''$. Since $X^*\cap X''=\emptyset$ we get by the definition of~$X^*$ and~$X^{**}$
    \begin{equation}\begin{split}\label{eq:twoside:upperT}
      T &\le \Big(|X^{**}|(d-\eps^*)p+|X''\setminus X^{**}|(d+(\eps^*)^2p)\Big)
               (1+\eps)^2p^2|Y'||Z| \\
       &= \Big(\mu(d-\eps^*)
       +(1-\mu)(d+(\eps^*)^2)\Big)(1+\eps)^2p^3|X''||Y'||Z|\\
        &\le \big(d-\mu\eps^*+(\eps^*)^2+3\eps\big)p^3|X''||Y'||Z| \,.
    \end{split}\end{equation}

    For obtaining a lower bound on~$T$, let $Y''\subset Y'$ be the set of
    vertices $y\in Y'$ with 
    \[\deg_G(y;Z)\ge(d-\eps)p|Z|\quad\text{and}\quad\deg_\Gamma(y;X'')\ge(1-\eps)p|X''|\,.\]
    Since $(Y,Z)$ is $(\eps,d,p)$-regular in~$G$ and $(X'',Y)$ is
    $(p,2cp^2\sqrt{|X''||Y|})$-bijumbled in~$\Gamma$, 
    applying Lemma~\ref{lem:degbound} we obtain
	$|Y'\setminus Y''|\le|\eps|Y|+8c^2p^2\eps^{-2}|Y|\le 2\eps|Y|$
    by~\eqref{eq:twoside:epsc}.
    Now, each $y\in Y''$ contributes at least $T(y):=e_\Gamma\big(N_\Gamma(y,X''),N_G(y,Z)\big)$
    triples to~$T$. As $(X,Z)$ is $(p,cp^3\sqrt{|X||Z|})$-bijumbled
    the definition of~$Y''$ thus implies that
    \begin{align*}
      T(y)
      &\ge p\cdot (1-\eps)p|X''|(d-\eps)p|Z|-cp^3\sqrt{|X||Z|}\sqrt{(1-\eps)p|X''|(d-\eps)p|Z|}\\
      &\geByRef{eq:twoside:epsc}(d-3\eps)p^3|X''||Z| \,,
    \end{align*}
      for each $y\in Y''$.
    We conclude that
    \begin{equation*}
      T\ge\sum_{y\in Y''} T(y)\ge \big(|Y'|-2\eps|Y|\big)(d-3\eps)p^3|X''||Z|\ge (d-10\eps)p^3|X''||Y'||Z|\,.
    \end{equation*}
    Together with~\eqref{eq:twoside:upperT}  this gives
    $d-\mu\eps^*+(\eps^*)^2+3\eps\ge d-10\eps$
    and so $\mu\le \eps^*+13\eps(\eps^*)^{-1}\le 2\eps^*$
    by~\eqref{eq:twoside:epsc} as desired.
  \end{claimproof}
  Claim~\ref{clm:twoside:rightdensity} and~\eqref{eq:twoside:X*} prove the lemma.
\end{proof}


\bibliographystyle{amsplain_yk} 
\bibliography{RegInherit}


\appendix

\section{Counting Lemmas} 
\label{sec:counting}

In this appendix we formulate a sparse one-sided Counting Lemma and a
sparse two-sided Counting Lemma (requiring stronger bijumbledness), which
both follow from our Inheritance Lemmas.

Given a graph $H$ with $V(H)=[m]$, a graph $G$, and vertex subsets $V_1,\dots,V_m$ of $V(G)$, we write $n(H;G)$ for the number of labelled copies of $H$ in $G$ with $i$ in $V_i$ for each $i$. Observe that the quantity $n(H;G)$ depends on the choice of the sets $V_1,\dots,V_m$, but this choice will always be clear from the context. Given $0<p\le 1$, we write
\[d(H;G):=\prod_{ij\in E(H)}d_p(V_i,V_j)\,.\]
Again, this quantity depends on the choice of $V_1,\dots,V_m$, and again this will always be clear from the context.

Still for any given graph $H$ with vertex set $[m]$, which we think of as having order $1,\dots,m$, and given $u,v\in [m]$, we define
\begin{align*}
 N^+(v)&:=\big\{w\in N_H(v)\colon w>v\big\}\,,\\
 N^-(v)&:=\big\{w\in N_H(v)\colon w<v\big\}\,,\\
 N^{<u}(v)&:=\big\{w\in N_H(v)\colon w<u\big\}\,.
\end{align*}

Finally, we let $\kreg(H)$  be the smallest
number with the following properties for each $1\le i\le m$.
For each $j\ge i$ such that $ij\in E(H)$
\begin{equation*}
 \kreg(H)\ge\tfrac12|N^{-}(i)|+\tfrac12|N^{<i}(j)|+\begin{cases}
   \tfrac{3}{2} & \text{ if }\exists k>i\colon jk\in E(H)\\
  2 & \text{ if }\exists k>i\colon jk,ik\in E(H)\\
  3 & \text{ if }\exists k>i\colon jk,ik\in E(H) \text{ and } \\
     &                \qquad\qquad\qquad |N^{<i}(k)|\le |N^{<i}(j)| \\
  1 & \text{ otherwise}
\end{cases}
\end{equation*}
and for each $j,j'\ge i$ such that $ij,jj'\in E(H)$
\begin{equation*}
  \kreg(H)\ge\tfrac12|N^{<i}(j)|+\tfrac12|N^{<i}(j')|+\begin{cases}
    2.501 & \text{ if }ij'\in E(H) \,, \\
    2.001 & \text{otherwise} \,.
  \end{cases}
\end{equation*}
Informally, the idea is that $(p,cp^{\kreg(H)})$-bijumbledness is enough to use Lemmas~\ref{lem:oneside} and~\ref{lem:twoside} to find copies of $H$ in $G$ one vertex at a time, in the natural order $1,\dots,m$. The following lemma formalises this.

\begin{lemma}[One-sided Counting Lemma]\label{lem:onecount}
  For every graph $H$ with $V(H)=[m]$ and every $\gamma>0$, there exist $\eps,c>0$ such that the
  following holds.  Let $G$ and $\Gamma$ be graphs with
  $G\subset\Gamma$, and let $V_1,\ldots,V_m$ be subsets of $V(G)$. Suppose that for each edge $ij\in H$, the sets $V_i$ and $V_j$ are disjoint, and the pair $(V_i,V_j)$ is $(\eps,p)$-regular in $G$ and $(p,cp^{\kreg(H)}\sqrt{|V_i||V_j|})$-bijumbled in~$\Gamma$.
  Then we have
  \[n(H;G)\ge \big(d(H;G)-\gamma\big)p^{e(H)}\prod_{i\in V(H)}|V_i|\,.\] 
\end{lemma}
The proof of this lemma is similar to the proof of~\cite[Lemma X]{CFZ14}. It is also contained in the proof of Lemma~\ref{lem:twocount} below, so we omit the details.

\medskip

The jumbledness requirement in our two-sided Counting Lemma depends on
another graph parameter, which is also different from the parameter
 in the two-sided Counting Lemma in~\cite{CFZ14} and may appear
somewhat exotic at first sight. We shall later compare this parameter to
other more common graph parameters.

Let $H$ be given with vertex set $[m]$, which again we think of as having the order $1,\dots,m$.
For each $v\in V(H)$, let $\tau_v$ be any ordering of $N^+(v)$ such that $|N^{<v}(w)|$ is decreasing. We define
\[\degt(H):=\max_{v\in V(H)}\Big(|N^-(v)|+\max_{w\in N^+(v)}\big(\tau_v(w)+|N^{<v}(w)|\big)\Big)\,.\]

The idea is that this parameter controls the bijumbledness we require in order to prove an upper bound on the number copies of $H$ in $\Gamma$. In order to count in $G$, we need both to be able to do this and to use our inheritance lemmas, and we need to consider the same order on $V(H)$ for both.

\begin{lemma}[Two-sided Counting Lemma]\label{lem:twocount}
  For every graph $H$ with $V(H)=[m]$ and every $\gamma>0$, there exist $\eps,c>0$ such that the
  following holds. We set
  \begin{equation}\label{eq:countind:beta}
   \beta=cp^{\max\big(\kreg(H),\tfrac12+\tfrac12\degt(H)\big)}\,.
  \end{equation}
  Let $G$ and $\Gamma$ be graphs with
  $G\subset\Gamma$, and let $V_1,\ldots,V_m$ be subsets of $V(G)$. Suppose that for each edge $ij\in H$, the sets $V_i$ and $V_j$ are disjoint, and the pair $(V_i,V_j)$ is $(\eps,p)$-regular in $G$ and $(p,\beta\sqrt{|V_i||V_j|})$-bijumbled in~$\Gamma$.
  Then we have
  \[n(H;G)= \big(d(H;G)\pm\gamma\big)p^{e(H)}\prod_{i\in V(H)}|V_i|\,.\] 
\end{lemma}

As with Lemma~\ref{lem:onecount}, in applications one should choose the order on $V(H)$ so that the resulting $\beta$ is as large as possible.

For comparison to more standard graph parameters, observe in an optimal order we have
\[\tfrac{\Delta(H)+1}{2}\le\tfrac12\degt(H)+\tfrac12\le\tfrac{\Delta(H)+\degen(H)}{2}\,,\]
where $\degen(H)=\min\{d:\forall H'\subset H,\delta(H')\le d\}$ is the \emph{degeneracy} of $H$. To see that the former inequality is true, observe that for any $v\in V(H)$ we have
\[|N^-(v)|+\max_{w\in N^+(v)}\big(\tau_v(w)+|N^{<v}(w)|\big)\ge|N^-(v)|+|N^+(v)|=d(v)\,,\]
and thus, $\tilde d(H)\geq\Delta (H)$.
For the latter, consider a degeneracy order on $H$, that is, an order in which each vertex has at most $\degen(H)$ neighbours preceding it. For such an order, for any $v$ and any $w\in N^+(v)$, we have $|N^{<v}(w)|\le\degen(H)-1$, since $N^{<v}(w)$ contains neighbours of $w$ preceding $w$, but not including $v$. We thus have, for each $v\in V(H)$,
\[|N^-(v)|+\max_{w\in N^+(v)}\big(\tau_v(w)+|N^{<v}(w)|\big)\le |N^-(v)|+|N^+(v)|+\degen(H)-1\,.\]

In the similar two-sided counting result of~\cite{CFZ14}, the exponent of $p$ in bijumbledness is
\[\min\big(\tfrac{\Delta(L(H))+4}{2},\tfrac{\degen(L(H))+6}{2}\big)\,,\]
where $L(H)$ is the \emph{line graph} of $H$, namely the graph with vertex set $E(H)$ in which two vertices are adjacent if they are incident as edges of $H$. It is easy to check that this parameter is bounded between $\tfrac{\Delta(H)+3}{2}$ and $\tfrac{\Delta(H)+\degen(H)+4}{2}$ (and both bounds can be sharp).

\medskip

We now briefly outline the proof of Lemma~\ref{lem:twocount}. We prove this statement by induction. We count the number of copies of $H$ in $G$, which is a subgraph of the bijumbled $\Gamma$, by embedding $H$ one vertex at a time, and bounding the number of choices at each step. Most of the time, we will choose to embed to vertices which maintain regularity, and thus we can accurately estimate the number of choices. This part of the proof is very similar to the usual proof of the Counting Lemma for dense graphs, except that we use Lemmas~\ref{lem:oneside} and~\ref{lem:twoside} to argue that regularity is usually maintained rather than this being a triviality. To deal with the exceptional event that we embed to a vertex and regularity is lost, we require an upper bound on $H$-copies in $\Gamma$. This is the content of the following Lemma~\ref{lem:badbound}.

Given $H$ with $V(H)=[m]$ in an order realising $\degt(H)$, and $x\in V(H)$, let \[H^{\ge x}:=H\big[\big\{y\in V(H)\colon y\ge x\big\}\big]\,.\]

\begin{lemma}\label{lem:badbound}
  Given $H$ with vertex set $[m]$ and $0<p<\tfrac1{10}$, suppose
  \[\beta\le \tfrac12\eps(50\Delta(H))^{-\Delta(H)} p^{\tfrac12+\tfrac12\degt(H)}\,.\]
  Let $V_1,\ldots,V_m$ be subsets of $V(\Gamma)$, and suppose $(V_i,V_j)$
  is $(p,\beta\sqrt{|V_i||V_j|})$-bijumbled in $\Gamma$ for each $ij\in
  E(H)$. Let $x\in V(H)$, and for each $y\in V\big(H^{\ge x}\big)$, let
  $W_y\subset V_y$ satisfy $|W_y|\ge\eps p^{|N^{<x}(y)|}|V_y|$. Then the
  number of copies of $H^{\ge x}$ in $\Gamma$ with $y$ in the set $W_y$ for
  each $y$ is at most
  \[(4p)^{e(H^{\ge x}_\pi)}\prod_{x\le y\le m}|W_y|\,.\]
\end{lemma}

We now show how this implies Lemma~\ref{lem:twocount}

\begin{proof}[Proof of Lemma~\ref{lem:twocount}]
  Suppose that $\Gamma$, $G$ and $H$ are as in the lemma statement. We will prove by induction the following statement $(\dagger)$.

   For every $\gamma'>0$ there exist $\eps',c'>0$ with the following
   property. Given $x\in V(H)$, for each $y\in V\big(H^{\ge x}\big)$, let
   $W_y\subset V_y$ satisfy $|W_y|\ge\eps' p^{|N^{<x}(y)|}|V_y|$. Suppose
   that for each $ij\in E(H)$ with $i,j\ge x$, the pair $(W_i,W_j)$ has
   $p$-density $d_p\big(W_i,W_j\big)\ge\gamma'$, is
   $\big(\eps',d_p(V_i,V_j),p\big)$-regular in $G$,
   and $(p,\beta\sqrt{|V_i||V_j|})$-bijumbled in~$\Gamma$ with
   $\beta=c'p^{\max\big(\kreg(H),\tfrac12+\tfrac12\degt(H)\big)}$.
   Then the number of copies of $H^{\ge x}$ in $G$ with $y\in W_y$ for each $y$ 
   is
   \[
     \big(d(H^{\ge x};G)\pm\gamma'\big)p^{e(H^{\ge x})}
     \prod_{x\le i\le m}|W_i|\,.
  \] 
  
   Before we prove this, we show that it implies the statement of
   Lemma~\ref{lem:twocount}. To that end, we assume $(\dagger)$ holds with
   input $\gamma'=\tfrac\gamma4$ and $x=1$, returning constants $c'$ and
   $\eps'$. We set $c=c'$ and $\eps=\tfrac12\eps'$, and use $W_y:=V_y$ for
   each $y\in V(H)$. 
  
   If for each $ij\in E(H)$ we have $d_p\big(V_i,V_j\big)\ge\tfrac\gamma4$,
   then by~$(\dagger)$ the lemma statement follows. We may therefore assume
   that there is some $ij\in E(H)$ such that
   $d_p\big(V_i,V_j\big)<\tfrac\gamma4$, and thus
   $d(H;G)<\tfrac\gamma2$. To establish the lemma statement it thus
   suffices to show
   \[n(H;G)\le\gamma p^{e(H)}\prod_{1\le i\le m}|V_i|\,.\] We generate a
   graph $G'$ with $G\subset G'\subset\Gamma$ by adding edges of $\Gamma$
   to each $(\eps,p)$-regular pair $(V_i,V_j)$ with
   $d_p(V_i,V_j;G)<\tfrac\gamma4$. We do this by choosing edges of $\Gamma$
   in such pairs uniformly at random with probability
   $\tfrac{3\gamma}{8}-d_p(V_i,V_j)$. We claim that the result is that any
   such pair $(V_i,V_j)$ is $(2\eps,p)$-regular in $G'$ with density
   between $\tfrac\gamma4p$ and $\tfrac{\gamma}{2}p$. The proof of this
   claim is a standard application of the Chernoff bound, and we omit
   it. Since there is a pair in $G'$ with density less than
   $\tfrac\gamma2p$, we have $d(H;G')\le\tfrac{3\gamma}{4}$. By
   construction we have $n(H;G)\le n(H;G')$, and by~$(\dagger)$ the lemma
   statement follows.
	
   We now prove~$(\dagger)$ by induction on $x$. The base case $x=m$ is
   trivial, with $\eps''=c''=1$, so we assume $x<m$. Given $\gamma'$, we set
   $\gamma''=\tfrac{\gamma'}2$. Let $\eps''$ and $c''$ be returned by~$(\dagger)$
   for input $\gamma''$ and $x+1$. Without loss of generality we assume
   $\eps''<4^{-m^2}\gamma'/(24qm)$. We set $\eps'>0$ small enough for
   Lemmas~\ref{lem:oneside} and~\ref{lem:twoside} with input
   $\tfrac12\gamma'\eps''$ and $\gamma'$, and such that
   $(1+\gamma'^{-1}\eps')^m<1+\tfrac{\gamma'}8$. We suppose that
   $c'\le\min\big(\eps'^3,c''\big)$ is small enough for both these
   applications and for Lemma~\ref{lem:badbound}. When $N_{H^{\geq x}}(x)=\emptyset$,
    statement~$(\dagger)$ for $x$ and $\gamma'$ follows trivially by~$(\dagger)$ with input $x+1$
   and $\gamma'$. Thus, suppose $N_{H^{\ge x}}(x)=\{y_1,\ldots,y_q\}$ for some $q\ge
   1$.
 
 If $v\in W_x$ fails to satisfy any of the following conditions, we say $v$ is bad.
 \begin{enumerate}[label=\abc]
  \item\label{twocount:1} For each $i$ we have $\deg_G(v,W_{y_i})= (d_p(V_x,V_{y_i})\pm \eps')p|W_{y_i}|$.
  \item\label{twocount:2} For each $i$ we have $\deg_\Gamma(v,W_{y_i})= (1\pm \eps')p|W_{y_i}|$.
  \item\label{twocount:3} For each $i$ and $z>x$ such that $y_iz\in
    E(H^{\ge x})$, the pair $\big(N_G(v,W_{y_i}),W_z)\big)$ is  $\big(\eps'',d_p(V_{y_i},V_z),p\big)$-regular in $G$.
  \item\label{twocount:4} For each $i\neq j$ such that $y_iy_j\in E(H^{\ge x})$ the pair $\big(N_G(v,W_{y_i}),N_G(v,W_{y_j})\big)$ is  $\big(\eps'',d_p(V_{y_i},V_{y_j}),p\big)$-regular in $G$.
 \end{enumerate}
 
 Let $B\subset W_x$ be the set of bad vertices. We now show that $B$ is much smaller than $W_x$. Since for each $1\le i\le q$ the pair $\big(W_x,W_{y_i}\big)$ is $\big(\eps',d_p(V_1,V_{y_i}),p\big)$-regular, there are at most $2q\eps'|W_x|$ vertices in $W_x$ for which~\ref{twocount:1} fails.
 
 For the remaining estimates, it is convenient to estimate the bijumbledness of pairs $(W_i,W_j)$ for $i,j\ge x$. Specifically, since $(V_i,V_j)$ is $\big(p,\beta\sqrt{|V_i||V_j|}\big)$-bijumbled, and since $|W_i|\ge\eps' p^{|N^{<x}(i)|}|V_i|$, and $|W_j|\ge\eps' p^{|N^{<x}(j)|}|V_j|$, we conclude that
 \begin{equation}\label{eq:twocount:WiWj}
  (W_i,W_j)\text{ is }\Big(p,\eps'^{-1}p^{-\tfrac12|N^{<x}(i)|-\tfrac12|N^{<x}(j)|}\beta\sqrt{|W_i||W_j|}\Big)\text{-bijumbled}\,.
 \end{equation}
 
 Next, let $Z\subset W_x$ be the set of vertices with at least $(1+\eps')p|W_{y_i}|$ neighbours in $W_{y_i}$. By~\eqref{eq:twocount:WiWj}, we have
 \[(1+\eps')p|W_{y_i}||Z|-p|W_{y_i}||Z|\le\eps'^{-1}p^{-\tfrac12|N^{<x}(y_i)|-\tfrac12|N^{<x}(x)|}\beta\sqrt{|W_x||W_{y_i}||W_{y_i}||Z|}\,.\]
 Since $\beta\le c'p^{\kreg(H)}$, and by definition of $\kreg$, we conclude $|Z|\le\eps'|W_x|$. A similar argument applies to the set of vertices in $W_x$ with at most $(1-\eps')p|W_{y_i}|$ neighbours in $W_{y_i}$ in $\Gamma$, so there are at most $2q\eps'|W_x|$ vertices in $W_x$ for which~\ref{twocount:2} fails.
 
 We move on to the regularity statements. By~\eqref{eq:twocount:WiWj} and
 definition of $\kreg$, the bijumbledness requirements of
 Lemma~\ref{lem:oneside} are satisfied, so since
 $d_p(V_{y_i},V_z)\ge\gamma'$ and $\big(V_{y_i},V_z\big)$ is
 $(\eps',p)$-regular, the number of vertices $v\in W_x$ such that
 $\big(N_\Gamma(v,V_{y_i}),V_z\big)$ is not
 $\big(\tfrac12\gamma'\eps'',d_p(V_{y_i},V_z),p\big)$-regular is at most
 $qm\eps''|W_x|$. Similarly, the bijumbledness requirements of
 Lemma~\ref{lem:twoside} are satisfied. Again, the number of vertices $v\in
 W_x$ such that $\big(N_\Gamma(v,V_{y_i}),N_\Gamma(v,V_{y_j})\big)$ is not
 $\big(\tfrac12\gamma'\eps'',d_p(V_{y_i},V_{y_j}),p\big)$-regular is at
 most $qm\eps''|W_x|$. Now suppose that $v$ satisfies~\ref{twocount:1}
 and~\ref{twocount:2}. By choice of $\eps'$ and by Lemma~\ref{lem:slicing},
 if $\big(N_\Gamma(v,V_{y_i}),V_z\big)$ is
 $\big(\tfrac12\gamma'\eps'',d_p(V_{y_i},V_z),p\big)$-regular then
 $\big(N_G(v,V_{y_i}),V_z\big)$ is
 $\big(\eps'',d_p(V_{y_i},V_z),p\big)$-regular, if
 $\big(N_\Gamma(v,V_{y_i}),N_\Gamma(v,V_{y_j})\big)$ is
 $\big(\tfrac12\gamma'\eps'',d_p(V_{y_i},V_{y_j}),p\big)$-regular then
 $\big(N_G(v,V_{y_i}),N_G(v,V_{y_j})\big)$ is
 $\big(\eps'',d_p(V_{y_i},V_{y_j}),p\big)$-regular. We conclude that in
 total at most $2qm\eps''|W_x|$ vertices of $W_x$ which
 satisfy~\ref{twocount:1} and~\ref{twocount:2} fail either~\ref{twocount:3}
 or~\ref{twocount:4}, so $|B|\le 2qm(2\eps'+\eps'')|W_x|$. 

 Now given any $v\in W_x\setminus B$, we wish to estimate the number of
 copies of $H^{\ge x}$ in $G$ such that $x$ is mapped to $v$ and $i$ is in
 $W_i$ for each $x+1\le i\le m$. In other words, we need to know the number
 of copies of $H^{\ge x+1}$ such that $i$ is in $W'_i$ for each $x+1\le
 i\le m$, where $W'_i=W_i$ if $xi\not\in E(H)$ and $W'_i=N_G(v)\cap W_i$ if
 $xi\in E(H)$. Because $v$ satisfies~\ref{twocount:3} and~\ref{twocount:4},
 for each $ij\in E(H^{\ge x+1})$ the pair $(W'_i,W'_j)$ is
 $(\eps'',p)$-regular in $G$. We now use the induction hypothesis. By
 choice of $c'$, we can apply~$(\dagger)$ with input $x+1$ and
 $\tfrac{\gamma'}2$ to obtain
 \begin{align*}
  n(H^{\ge x+1};G)&=\big(d(H^{\ge x+1};G)\pm\tfrac{\gamma'}2\big)p^{e(H^{\ge x+1})}\prod_{x+1\le i\le m}|W'_i|\\
  &=\big(d(H^{\ge x};G)\pm\tfrac{\gamma'}2\big)p^{e(H^{\ge x})}(1\pm\eps'\gamma'^{-1})^q\prod_{x+1\le i\le m}|W_i|\,,
 \end{align*}
 where the second line uses the fact that $(W'_i,W'_j)$ is $\big(\eps'',d_p(V_i,V_j),p\big)$-regular for each $ij\in E(H^{\ge x+1})$ by~\ref{twocount:3} and~\ref{twocount:4}, and the fact $|W'_i|=\big(d_p(V_x,V_i)\pm\eps'\big)p|W_i|$ for each $i$ such that $xi\in E(H^{\ge x})$. We conclude that the number of copies of $H^{\ge x}$ in $G$ with $x$ mapped to $W_x\setminus B$ and $i$ mapped to $W_i$ for each $x+1\le i\le m$ is
 \begin{multline*}
  \big(|W_x|-|B|\big)\big(d(H^{\ge x};G)\pm\tfrac{\gamma'}2\big)p^{e(H^{\ge x})}(1\pm\eps'\gamma'^{-1})^q\prod_{x+1\le i\le m}|W_i|\\
  =\big(d(H^{\ge x};G)\pm\tfrac{3\gamma'}{4}\big)p^{e(H^{\ge x})}\prod_{x\le i\le m}|W_i|\,,
 \end{multline*}
 where the second line follows by choice of $\eps'$. This gives the desired lower bound; it only remains to complete the upper bound by showing that the number of copies of $H^{\ge x}$ in $G$ with $x$ in $B$ and $i$ in $W_i$ for each $x+1\le i\le m$ is at most
 \[\tfrac{\gamma'}4p^{e(H^{\ge x})}\prod_{x\le i\le m}|W_i|\,.\]
 
 We may assume $|B|=6qm\eps''|W_x|$  by, if necessary, adding arbitrary vertices of $W_x$.
 By choice of $c'$ and~\eqref{eq:countind:beta}, the jumbledness requirements of Lemma~\ref{lem:badbound} are satisfied, so by that lemma the number of copies of $H^{\ge x}$ in $\Gamma$ with $x$ in $B$ and $i$ in $W_i$ for each $x+1\le i\le m$ is at most
 \[(4p)^{e(H^{\ge x})}|B|\prod_{x\le i\le m}|W_i|\le\tfrac{\gamma'}4p^{e(H^{\ge x})}\prod_{x\le i\le m}|W_i|\,,\]
 where the inequality is true by choice of $\eps''$.
\end{proof}

It remains to prove Lemma~\ref{lem:badbound}. In this proof we will need to
optimise over certain configurations. This optimisation problem is captured
in the following lemma, where it is phrased as the problem of bounding a
certain sum of real numbers.

\begin{lemma}\label{lem:optialpha}
  Let $0<p\le \frac 1{10}$ be real, $q\ge 1$ be an integer and $b_1\ge\dots\ge b_q$ be
  non-negative integers. Let $P:=\lfloor\log(p^{-1})\rfloor$ and $C:=\max_{1\le
    i\le q}(b_i+i)$. Let $A:=[0,P]^q\setminus\{\mathbf{0}\}$ be the set of
  non-zero $q$-dimensional vectors with integer entries between~$0$ and~$P$.
  Then
 \begin{equation}\label{eq:optisum}
  \sum_{\bfalpha\in A}\frac{2^{\alpha_1+\dots+\alpha_q}}{\max_{i:\alpha_i\neq 0}2^{2\alpha_i}p^{b_i}}\le (50q)^qp^{1-C}\,.
 \end{equation}
\end{lemma}
\begin{proof}
  Given $\bfalpha\in A$ let
  \begin{equation}\label{eq:optiM}
    M(\bfalpha):=\frac{2^{\alpha_1+\dots+\alpha_q}}{\max_{i:\alpha_i\neq
        0}2^{2\alpha_i}p^{b_i}}\,.
  \end{equation}
  We first establish the following bounds on~$M(\bfalpha)$.
 \begin{claim}\label{cl:opti}
    For each $\bfalpha\in A$ one of the following holds.
   \begin{enumerate}[label=\rom]
     \item\label{opti:1} 
       $M(\bfalpha)\le 2^{-\alpha_1-\dots-\alpha_q}p^{-b_1}$.
     \item\label{opti:2} 
       $M(\bfalpha) \le p^{-C+\tfrac54}$.
   \end{enumerate}
 \end{claim}
\begin{claimproof}[Proof of Claim~\ref{cl:opti}]
  It simplifies our arguments to pass to an optimisation over real-valued
  variables: Let~$\tilde A$ be the set of non-zero vectors from $\mathbb{R}^q$ with
  entries between~$0$ and~$P$. Then, $M(\bfalpha)$ for
  $\bfalpha\in\tilde A$ is defined as in~\eqref{eq:optiM}.

 We prove the claim by finding for each~$\bfalpha\in A$ an $\bfalpha'\in
 \tilde A$ of simple structure such that $M(\bfalpha)$ and $M(\bfalpha')$
 are related suitably.  We construct $\bfalpha'$ by applying the following
 three operations successively until no further operation is possible. Each
 of these operations takes a vector $\tilde\bfalpha\in\tilde A$ and returns
 a new vector $\tilde\bfalpha'\in\tilde A$.

  The first operation applies when there are two entries $\tilde\alpha_i$
  and $\tilde\alpha_j$ such that $\tilde\alpha_i$ does not realise the
  maximum in $M(\tilde\bfalpha)$ (by this we mean the term $\max_{i:\alpha_i\neq
        0}2^{2\alpha_i}p^{b_i}$), but $\tilde\alpha_j$ does, and moreover, we have
  $2^{2\tilde\alpha_i}p^{b_i}<2^{2\tilde\alpha_j}p^{b_j}$ (which might not
  be the case if $\tilde\alpha_i=0$). Then we can
  increase $\tilde\alpha_i$ until it contributes the same value to the
  maximum as $\tilde\alpha_j$ or hits~$P$. More precisely,
  if there exist $i$ and $j$ such that $\tilde\alpha_i<P$, such that
 $\tilde\alpha_j>0$ and such that $2^{2\tilde\alpha_i}p^{b_i}<2^{2\tilde\alpha_j}p^{b_j}=\max_{k:\alpha_k\neq 0}2^{2\alpha_k}p^{b_k}$,
 then we can take $\tilde\bfalpha'$ equal to $\tilde\bfalpha$ at all entries except the
 $i$th, and \[\tilde\alpha'_i=\min\big(P,
 \tfrac12\log(2^{2\tilde\alpha_j}p^{b_j-b_i})\big)\,.\] Then the maximum in
 $M(\tilde\bfalpha')$ is equal to that in $M(\tilde\bfalpha)$ because
 \[2^{2\tilde\alpha'_i}p^{b_i}\le 2^{2\tilde\alpha_j}p^{b_j-b_i}p^{b_i}=2^{2\tilde\alpha_j}p^{b_j}=2^{2\tilde\alpha'_j}p^{b_j}\,.\]
 Moreover
 $\tilde\alpha'_i>\tilde\alpha_i$, so $\tilde\alpha'_1+\dots+\tilde\alpha'_q\ge\tilde\alpha_1+\dots+\tilde\alpha_q$ and 
 $M(\tilde\bfalpha')>M(\tilde\bfalpha)$.
 
 The second operation applies when there are two entries strictly between~$0$ and~$P$. In this case we can
 increase both entries by the same amount until one hits~$P$.
 More precisely, if there exist $i$ and $j$ with
 $0<\tilde\alpha_i\le\tilde\alpha_j<P$, then we can take $\tilde\bfalpha'$
 equal to $\tilde\bfalpha$ at all entries except the $i$th and $j$th. We
 set $\tilde\alpha'_j=P$, and
 $\tilde\alpha'_i=\tilde\alpha_i+P-\tilde\alpha_j$. Then the maximum in
 $M(\tilde\bfalpha')$ is greater than the one in $M(\tilde\bfalpha)$ by a
 factor of at most
 $\max\{2^{2(\tilde\alpha'_j-\tilde\alpha_j)},2^{2(\tilde\alpha'_i-\tilde\alpha_i)}\}=
 2^{2P-2\tilde\alpha_j}$, while the sum of the entries of $\tilde\bfalpha'$
 is greater by $2P-2\tilde\alpha_j$ than that of $\tilde\bfalpha$. So again
 $\tilde\alpha'_1+\dots+\tilde\alpha'_q\ge\tilde\alpha_1+\dots+\tilde\alpha_q$ and  $M(\tilde\bfalpha')\ge M(\tilde\bfalpha)$.

 The third operation is the only operation that decreases coordinates: It decreases a coordinate if it is the only coordinate
 realising the maximum in $M(\tilde\bfalpha)$ until it contributes as much
 to the maximum as some other coordinate or hits~$0$. More precisely,
 it applies if there are at least two non-zero entries in
 $\tilde\bfalpha$, and $j$ is the unique coordinate realising
 $\max_{i:\tilde\alpha_i>0}2^{2\tilde\alpha_i}p^{b_i}$. Let $s$ be the
 second greatest value of $2^{\tilde\alpha_i}p^{b_i}$ over $i$ such that
 $\tilde\alpha_i>0$, and set 
 \[c=\min\big(\tilde\alpha_j,\tfrac12\log(sp^{-b_j})\big)\,.\] Then let
 $\tilde\bfalpha'$ be equal to $\tilde\bfalpha$ in all coordinates except
 $\tilde\alpha'_j=\tilde\alpha_j-c$. The maximum in $M(\tilde\bfalpha')$ is
 less than that in $M(\tilde\bfalpha)$ by a factor of at least $2^{2c}$ (if
 $c=\tilde\alpha_j$ the factor may be greater).  We conclude that
 $\tilde\alpha'_1+\dots+\tilde\alpha'_q=\tilde\alpha_1+\dots+\tilde\alpha_q-
 c$ and $M(\tilde\bfalpha')\ge 2^cM(\tilde\bfalpha)$.

 Now consider the result $\bfalpha'$ of the successive application of these
 operations until none of them can be applied anymore. (Observe that eventually 
 such a final $\bfalpha'$ must be reached, since the way the three operations are 
 defined prevents the process from going on indefinitely.)
 Clearly $M(\bfalpha')\ge M(\bfalpha)$ since no operation decreases $M(\cdot)$.
 For the structure of $\bfalpha'$ there are two
 possibilities.  On the one hand, $\bfalpha'$ could have exactly one
 non-zero coordinate $\alpha_\ell$.  Then, since the third operation is the
 only operation that decreases coordinates, this operation was applied to
 each coordinate which was non-zero in~$\bfalpha$ but coordinate $\ell$.  As
 the first two operations increase the sum of coordinates of $\bfalpha$ and
 increase the value of $M(\cdot)$, and the third operation increases the
 value of $M(\cdot)$ by at least a factor of $2^c$, we conclude
 $M(\bfalpha)\le 2^{-\alpha_1-\dots-\alpha_q+\alpha'_\ell}M(\bfalpha')$.
Clearly
$M(\bfalpha')=2^{-\alpha'_\ell}p^{-b_\ell}\le 2^{-\alpha'_\ell}p^{-b_1}$,
and hence $M(\bfalpha)\le 2^{-\alpha_1-\dots-\alpha_q}p^{-b_1}$,
so~\ref{opti:1} is satisfied.

On the other hand, it could be that $\bfalpha'$ has at least two non-zero
coordinates. Observe that at most one entry of $\bfalpha'$ is not in
$\{0,P\}$ by the second operation. By the third operation there are also at
least two coordinates realising the maximum in $M(\bfalpha')$, one of which
has value~$P$.  If there is a coordinate $j$ such that $0<\alpha'_j<P$ then
this coordinate also realises the maximum by the first operation. So it
follows from $b_{i+1}\ge b_i$ that and the first operation that there is an
index $\ell\ge 1$ such that for all $i\le\ell$ we have $\alpha'_i=P$, for
all $i>\ell+1$ we have $\alpha'_i=0$, and $\alpha'_{\ell+1}<P$.  Now if
$\alpha'_{\ell+1}>0$, then it realises the maximum and thus
$2^{2\alpha'_{\ell+1}}p^{b_{\ell+1}}=2^{2P}p^{b_{\ell}}$. Since $b_{\ell}$
and $b_{\ell+1}$ are integers, this equation can only be solved if
$b_{\ell}-b_{\ell+1}$ is equal to $0$, $1$ or $2$. In the first case we
obtain $\alpha'_{\ell+1}=P$, contradicting the definition of $\ell$. In
both of the other two cases we have $\alpha'_{\ell+1}\le\frac34P$.
 It follows that $\alpha'_{\ell+1}\le\frac34P$.
 Clearly we have
 $\max_{i:\alpha_i\neq 0}2^{2\alpha_i}p^{b_i}=2^{2P}p^{b_\ell}$, and so
 \[M(\bfalpha)\le M(\bfalpha')
 \le\frac{2^{\ell P+\tfrac34 P}}{2^{2P}p^{b_\ell}}
 =2^{(\ell-\tfrac54)P}p^{-b_\ell}
 \le p^{-\ell+\tfrac54-b_\ell}\le p^{-C+\tfrac54}\,,\]
 because $C=\max_{1\le
    i\le q}(b_i+i)$.
 Hence~\ref{opti:2} holds. 
\end{claimproof}

Now consider first all~$\bfalpha\in A$ for which Claim~\ref{cl:opti}\ref{opti:1}
holds. 
Since $M(\bfalpha)\le 2^{-\alpha_1-\dots-\alpha_q}p^{-b_1}$,
the contribution to~\eqref{eq:optisum} of all such~$\bfalpha$ is at most
$\sum_{K=1}^{qP}(K+q)^q2^{-K}p^{-b_1}$,
where we used that the number of vectors in~$A$ whose entries sum to $K$
is at most $\binom{K+q-1}{q-1}<(K+q)^q$.
Since $(1+\frac{K}{10q})^q\le\exp(K/10)\le 2^K$ we have
$(K+q)^q2^{-K}\le(10q)^q$ and hence the contribution to~\eqref{eq:optisum}
is at most
\[\sum_{K=1}^{qP}(K+q)^q2^{-K}p^{-b_1}\le q(10q)^q p^{-b_1} \,.\]

Finally, consider all~$\bfalpha\in A$ for which Claim~\ref{cl:opti}\ref{opti:2} holds. We require a preliminary estimate. For all $z>1$ we have $1+\log_e z-z<0$, since this function is equal to zero at $z=1$ and has first derivative $\tfrac{1}{z}-1$ which is negative for all $z>1$. It follows that for any $z>1$, if $x\ge e^{4q}$ we have
\[\big(2\log_2(x^z)\big)^qx^{-z/4}\le\big(2\log x\big)^qx^{-1/4}z^q e^{q(1-z)}<\big(2\log x\big)^qx^{-1/4}\,,\]
where we used $1+\log_e z-z<0$ to establish the second inequality. It follows that
\[(2\log_2 x)^qx^{-1/4}\le (16q)^q\,,\]
holds for all $x\ge1$. To see that this is true, observe that the left hand side is trivially at most the claimed bound when $1\le x\le e^{4q}$ (since the term $x^{-1/4}$ is at most one), and strictly decreasing for $x\ge e^{4q}$ by the previous calculation.

Since $M(\bfalpha)\le p^{-C+\tfrac54}$
the contribution to~\eqref{eq:optisum} of such $\bfalpha$ is at most
\[(P+1)^q p^{-C+\frac54}\le (2\log(p^{-1}))^q p^{-C+\frac54} \le
(16q)^qp^{-C+1}\,,\]
where we used the above estimate for the second inequality. We obtain~\eqref{eq:optisum}.
\end{proof}
 
With the help of Lemma~\ref{lem:optialpha} we can now prove Lemma~\ref{lem:badbound}.

\begin{proof}[Proof of Lemma~\ref{lem:badbound}]
  We prove the statement by induction on $x$. The base case $x=m$ is
  trivial, so suppose that $1\le x\le m-1$, and for an induction hypothesis
  that the lemma statement holds for $x+1$.
 
  In the case $N^+(x)=\emptyset$, the statement follows by applying the
  induction hypothesis and the same sets $W_y$ for $y>x$. Thus
  we can assume $\big|N^+(x)\big|=q\ge 1$. Let $N^+(x)=\{y_1,\dots,y_q\}$
  in an order such that $\big|N^{<x}(y_i)\big|\ge \big|N^{<x}(y_j)\big|$
  whenever $i<j$.

  For a fixed $v\in W_x$ we obtain the following estimate of~$H^{\ge x}$
  copies using this vertex.  For $1\le i\le q$ set
  $W'_{y_i}:=N_\Gamma(v;W_{y_i})$, and possibly add some arbitrary vertices of
  $W_{y_i}$ to $W'_{y_i}$ until $|W'_{y_i}|\ge p|W_{y_i}|$.  For all
  $y\not\in N^+(x)$ with $y>x$ set $W'_y:=W_y$.  Then by induction, the
  number of copies of $H^{\ge x+1}$ in $\Gamma$ with $y$ mapped to $W'_y$ for each
  $y$ is at most $(4p)^{e(H^{\ge x})-q}\prod_{y>x}|W'_y|$. It follows that
  the number of copies of $H^{\ge x}$ with $x$ mapped to~$v$ and~$y$ mapped
  to $W_y$ for each~$y$ is at most
 \begin{equation}\label{eq:countHgx}
   (4p)^{e(H^{\ge x})-q}
   \Big( \prod_{1\le i\le q} \min \big\{p|W_{y_i}|,\deg_\Gamma(v;W_{y_i})\big\} \Big)
   \prod_{y>x,y\not\in N^+(x)}|W_y|\,.
 \end{equation}
 
 We next partition $W_x$ as follows. Given $\bfalpha\in[0,\lfloor\log p^{-1}\rfloor]^q$ with integer entries, we let $B_{\bfalpha}$ be the set of vertices $v\in W_x$ such that for each $1\le i\le q$, either $\alpha_i=0$ and we have $\deg_\Gamma(v,W_{y_i})\le 2p|W_{y_i}|$, or $\alpha_i>0$ and we have
 \[2^{\alpha_i}p|W_{y_i}|<\deg_\Gamma(v,W_{y_i})\le 2^{\alpha_i+1}p|W_{y_i}|\,.\]
 Note that this is a partition because $2^{\lfloor\log p^{-1}\rfloor+1}p>1$.

 Using $|B_{\mathbf{0}}|\le |W_x|$ and~\eqref{eq:countHgx} with
 $\deg_\Gamma(v;W_{y_i})\le 2p|W_{y_i}|$ for each~$i$, we can immediately bound the number of copies
 of $H^{\ge x}$ in $\Gamma$ with $x$ mapped to $B_{\mathbf{0}}$ and $y$
 mapped to $W_y$ for each $y>x$ from above by
 \begin{equation}\label{eq:contribb0}
  |W_x|(4p)^{e(H^{\ge x})-q}(2p)^q\prod_{y>x}|W_y|\le\tfrac12
  (4p)^{e(H^{\ge x})}\prod_{x\le y\le m}|W_y|\,.
 \end{equation}

 It remains to establish an analogous bound for the sets $B_{\bfalpha}$
 with $\bfalpha\neq\mathbf{0}$. For this we use the jumbledness of
 $\Gamma$. Given $\bfalpha$ and some $1\le i\le q$ such that $\alpha_i\neq
 0$, we have $e\big(B_{\bfalpha},W_{y_i}\big)\ge 2^{\alpha_i}p|B_{\bfalpha}||W_{y_i}|$,
 and since $\alpha_i\ge 1$ it follows that
 \[e\big(B_{\bfalpha},W_{y_i}\big)-p\big|B_{\bfalpha}\big|\big|W_{y_i}\big|\ge\tfrac12\cdot
 2^{\alpha_i}p|B_{\bfalpha}||W_{y_i}|\,.\] Since $\big(V_x,V_{y_i}\big)$ is
 $\big(p,\beta\sqrt{|V_{x}||V_{y_i}|}\big)$-bijumbled, this implies
 \[\tfrac12\cdot 2^{\alpha_i}p|B_{\bfalpha}||W_{y_i}|\le \beta\sqrt{|V_{x}||V_{y_i}|}\sqrt{|B_{\bfalpha}||W_{y_i}|}\,.\]
 Rearranging this we obtain
 \[|B_{\bfalpha}|\le
 \frac{4\beta^2|V_{x}||V_{y_i}|}{2^{2\alpha_i}p^2|W_{y_i}|}\,. \] Since
 this holds for each $i$ with $\alpha_i>0$, using~\eqref{eq:countHgx} with
 $\deg_\Gamma(v;W_{y_i})\le 2^{\alpha_i+1}p|W_{y_i}|$ for each $i$, the number of
 $\phi$-partite copies of $H^{\ge x}$ in $\Gamma$ with $x$ in
 $B_{\bfalpha}$ and $y\in W_y$ for each $y>x$ is at most
 \begin{multline*}
   \Big(\min_{i:\alpha_i>0}
   \frac{4\beta^2|V_{x}||V_{y_i}|}{2^{2\alpha_i}p^2|W_{y_i}|}\Big)
   (4p)^{e(H^{\ge x})-q} \Big(\prod_{i=1}^q 2^{\alpha_i+1}p\Big)\prod_{y>x}|W_y|\\
   \le
   \Big(\min_{i:\alpha_i>0}\frac{4\beta^2|V_{x}||V_{y_i}|}{2^{2\alpha_i}p^2|W_{y_i}|}\Big)\Big(\prod_{i=1}^q
   2^{\alpha_i+1}\Big)4^{-q}(4p)^{e(H^{\ge x})}\prod_{y>x}|W_y|\,.
 \end{multline*}
 Since $|W_{y_i}|\ge\eps p^{|N^{<x}(y_i)||V_{y_i}|}$ this is at most
 \[
     \frac{2^{\alpha_1+\dots+\alpha_q}}{\max_{i:\alpha_i>0}2^{2\alpha_i}p^{|N^{<x}(y_i)|}}\cdot
     4\beta^2\eps^{-1}|V_{x}|p^{-2}2^q\cdot 4^{-q}(4p)^{e(H^{\ge x})}\prod_{y>x}|W_y|\,.
  \]
  By Lemma~\ref{lem:optialpha}, with $b_i=|N^{<x}(y_i)|$ for each $1\le
  i\le q$, letting $C=\max_{i=1}^q(b_i+i)$, the sum of these terms over all
  $B_{\bfalpha}$ with $\bfalpha\neq\mathbf{0}$ is at most
  \[(50q)^qp^{1-C}\cdot 4\beta^2\eps^{-1}|V_{x}|p^{-2}(4p)^{e(H^{\ge x})}\prod_{y>x}|W_y|\,.\]
  We have 
  $|V_x|\le\eps^{-1} p^{-|N^-(x)|}|W_{x}|$, and $C+|N^-(x)|\le\degt(H)$ by definition, so this is bounded above by
  \[4\beta^2\eps^{-2}(50q)^qp^{-1-\degt(H)}(4p)^{e(H^{\ge x})}\prod_{x\le
    y\le m}|W_y|\le\tfrac12(4p)^{e(H^{\ge x})}\prod_{x\le y\le m}|W_y|\,,\]
  where the inequality is by choice of $\beta$. Together
  with~\eqref{eq:contribb0} we obtain the claimed upper bound.
\end{proof}

\end{document}